\documentclass[12pt]{article}

\usepackage{amsmath}
\usepackage[colorlinks=true]{hyperref}
\usepackage{amssymb}
\usepackage{amsthm}
\usepackage{latexsym}
\usepackage{cite}
\usepackage{psfrag}
\usepackage{color}
\usepackage{amsfonts}
\usepackage{mathrsfs}
\usepackage{url}
\usepackage{graphicx}
\hypersetup{urlcolor=blue, citecolor=red}

\makeatletter
\@addtoreset{equation}{section}
\makeatother

\newtheorem{theorem}{Theorem}[section]
\newtheorem{corollary}[theorem]{Corollary}

\newtheorem{proposition}[theorem]{Proposition}
\newtheorem{lemma}[theorem]{Lemma}
\theoremstyle{definition}

\newtheorem{remark}[theorem]{Remark}

\newtheorem{example}[theorem]{Example}

\newcommand{\A}{\mathcal{A}}

\newcommand{\D}{\mathcal{D}}

\newcommand{\K}{\mathcal{K}}

\newcommand{\Pc}{\mathcal{P}}

\newcommand{\T}{\mathcal{T}}
\newcommand{\V}{\mathcal{V}}

\newcommand{\BB}{\mathfrak{B}}
\newcommand{\Dk}{\mathfrak{D}}

\newcommand{\F}{\mathbb{F}}
\newcommand{\G}{\mathbb{G}}

\newcommand{\U}{\mathscr{U}}

\newcommand{\PG}{\mathrm{PG}}

\newcommand{\db}{\displaybreak[3]}

\begin{document}

\title
{Further results on covering codes with radius $R$ and codimension $tR+1$
\date{}
\thanks{The research of S. Marcugini and F. Pambianco was supported in part by the Italian
National Group for Algebraic and Geometric Structures and their Applications (GNSAGA -
INDAM) (Contract No. U-UFMBAZ-2019-000160, 11.02.2019) and by University of Perugia
(Project No. 98751: Strutture Geometriche, Combinatoria e loro Applicazioni, Base Research
Fund 2017--2019; Fighting Cybercrime with OSINT, Research Fund 2021).}
}
\maketitle
\begin{center}
{\sc Alexander A. Davydov}\\
 {\sc\small Kharkevich Institute for Information Transmission Problems}\\
 {\sc\small Russian Academy of Sciences,
Moscow, 127051, Russian Federation}\\
 \emph{E-mail address:} alexander.davydov121@gmail.com\medskip\\
 {\sc Stefano Marcugini and
 Fernanda Pambianco }\\
 {\sc\small Department of  Mathematics  and Computer Science,  Perugia University,}\\
 {\sc\small Perugia, 06123, Italy}\\
 \emph{E-mail address:} \{stefano.marcugini, fernanda.pambianco\}@unipg.it
\end{center}

\textbf{Abstract.}
The length function $\ell_q(r,R)$ is the smallest
possible length $n$ of a $ q $-ary linear $[n,n-r]_qR$ code with codimension (redundancy) $r$ and covering radius $R$. Let $s_q(N,\rho)$ be the smallest size of a $\rho$-saturating set in the projective space $\mathrm{PG}(N,q)$. There is a one-to-one correspondence between $[n,n-r]_qR$ codes and $(R-1)$-saturating $n$-sets in $\mathrm{PG}(r-1,q)$ that implies $\ell_q(r,R)=s_q(r-1,R-1)$. In this work, for $R\ge3$,  new asymptotic upper bounds on $\ell_q(tR+1,R)$ are obtained in the following form:
\begin{align*}
&\bullet~\ell_q(tR+1,R) =s_q(tR,R-1)\\
    &\hspace{0.4cm} \le\sqrt[R]{\frac{R!}{R^{R-2}}}\cdot q^{(r-R)/R}\cdot\sqrt[R]{\ln q}+o(q^{(r-R)/R}), \hspace{0.3cm} r=tR+1,~t\ge1,\\
      &\hspace{0.4cm}~ q\text{ is an arbitrary prime power},~q\text{ is large enough};\\
       &\bullet~\text{ if additionally }R\text{ is large enough, then }\sqrt[R]{\frac{R!}{R^{R-2}}}\thicksim\frac{1}{e}\thickapprox0.3679.
    \end{align*}
The new bounds are essentially better than the known ones.

For $t=1$, a new construction of $(R-1)$-saturating sets in the projective space $\mathrm{PG}(R,q)$, providing sets of small sizes, is proposed. The $[n,n-(R+1)]_qR$  codes, obtained by the construction,  have minimum distance $R + 1$, i.e. they are almost MDS (AMDS) codes. These codes  are taken as the starting ones
in the lift-constructions (so-called ``$q^m$-concatenating constructions'') for covering codes to obtain infinite families of codes with growing codimension $r=tR+1$, $t\ge1$. 

\textbf{Keywords:} Covering codes, The length function, Saturating sets in projective spaces

\textbf{Mathematics Subject Classification (2010).} 94B05, 51E21, 51E22

\section{Introduction}\label{sec:Intr}
Let $\F_{q}^{\,n}$ be the $n$-dimensional vector space over the Galois field $\F_{q}$ with $q$ elements.
The sphere of radius $R$ with center $c$ in $\F_{q}^{\,n}$
is the set $\{v:v\in \F_{q}^{\,n},$ $d(v,c)\leq R\}$ where $d(v,c)$ is the Hamming distance between the vectors $v$ and $c$.
A linear code in $\F_{q}^{\,n}$
 with covering radius $R$, codimension (redundancy) $r$, and minimum distance $d$ is an $[n, n- r,d]_qR$ code. If $d$ is not relevant it can be omitted.
The value $R$  is the smallest integer such that the space $\F_{q}^{\,n}$ is covered by
the spheres of radius $R$ centered at the codewords.
Every vector in $\F_{q}^{\,r}$ is equal to a linear combination of at most $R$ columns of a parity check matrix of the code.
For an introduction to coding theory, see \cite{Bier,HufPless,MWS}. An $[n, n- r,d]_qR$ code with $d=r$ is an almost MDS (AMDS) code, see e.g.
\cite{AlderBruen-AMDS,Boer-AMDS,DodunLang-NMDS,AMDS-MenPelSal} and the references therein.

The minimum possible length $n$ such that an $[n, n-r ]_q R$ code exists is called \emph{the length function} and is denoted by $\ell_q(r,R)$. If $R$ and
$r$ are fixed, then the covering problem for codes is finding codes of small length. Codes investigated from the point
of view of the covering problem are called \emph{covering codes}. Studying covering codes is a classical combinatorial problem. Covering codes are connected with many theoretical and applied areas, see e.g. \cite[Section 1.2]{CHLS-bookCovCod}, \cite[Introduction]{DMP-AMC2021}, \cite{Graismer-2023}, and the references therein. For an introduction to  covering
codes, see \cite{Handbook-coverings,CHLS-bookCovCod,DGMP-AMC,Struik,LobstBibl,GrahSlo}.

This paper is devoted to the asymptotic upper bound on the length function $\ell_q(tR+1,R)$, $t\ge1$, when $q$ is a large enough arbitrary prime power.

Let $\PG(N,q)$ be the $N$-dimensional projective space over the Galois field~$\mathbb{F}_q$.
 We will say ``$N$-space'' (or ``$M$-subspace'') when the value of $q$ is clear by the context;
$M$ points of $\PG(N,q)$ are said to be \emph{in general position} if they generate an $(M-1)$-subspace.  A point of $\PG(N,q)$ in homogeneous coordinates can be considered as a vector of $\F_q^{N+1}$. Points in general position correspond to linear independent vectors.

Effective methods to obtain upper bounds on the length function $\ell_q(r,R)$ are connected with \emph{saturating sets in $\PG(N,q)$}.
 A point set $S\subseteq\PG(N,q)$ is
$\rho$-\emph{saturating} if any point $A\in\PG(N,q)$ lies in a $\rho$-subspace of $\PG(N,q)$ generated by $\rho+1$ points of $S$
  and $\rho$ is the smallest value with this property.
 Every
point $A\in\PG(N,q)$ can be written as a linear combination of at most $\rho+1$ points of $S$.
 In the literature, saturating sets are also called ``saturated
sets'', ``spanning sets'', and ``dense sets''.

Let $s_q(N,\rho)$ be the smallest size of a $\rho$-saturating set in $\mathrm{PG}(N,q)$.
If the positions of a column of a parity check matrix of an $[n,n-r]_qR$ code are treated as homogeneous coordinates of a point in $\PG(r-1,q)$, then this matrix is an $(R-1)$-saturating $n$-set in $\PG(r-1,q)$, and vice versa. So, there is a \emph{one-to-one correspondence} between $[n,n-r]_qR$ codes and $(R-1)$-saturating $n$-sets in $\PG(r-1,q)$. This implies
\begin{equation}\label{eq1:s=ell}
  \ell_q(r,R)=s_q(r-1,R-1).
\end{equation}
    For an introduction to geometries over finite fields and their connections with coding theory, see \cite{Bier,EtzStorm2016,Giul2013Survey,Hirs,HirsStor-2001,LandSt,DGMP-AMC} and the references therein.

Throughout the paper, $c$ is a constant  independent of $q$ but it is possible that $c$ is dependent on $r$ and $R$. In the latter case, $R$ can be used as a subscript of $c$. Also, the superscripts ``new" and ``knw" (i.e. ``known") are possible.

In \cite{BDGMP-R2R3CC_2019,DMP-R=3Redun2019}, \cite[Proposition 4.2.1]{denaux}, see also the references therein, lower bounds of the following form are considered:
\begin{equation}\label{eq1:lowbound}
\ell_q(r,R)\ge cq^{(r-R)/R},~  R\text{ and }r\text{ fixed}.
\end{equation}
In \cite{DGMP-AMC}, the bound \eqref{eq1:lowbound} is given in another \mbox{(asymptotic)} form.

In the literature, the bound \eqref{eq1:lowbound} is achieved for special values of $r,R,q$:
 \begin{align*}
 &r\ne tR,~  q=(q')^R~ \text{\cite{DGMP-AMC,denaux,denaux2,HegNagy}};~~R=sR',~ r=tR+s, ~q=(q')^{R'}~\text{\cite{DGMP-AMC,DMP-R=tR2019}};\db\\
 &r=tR,~q\text{ is an arbitrary prime power}~ \text{\cite{DavOst-IEEE2001,DGMP-AMC,DMP-R=tR2019,DavOst-DESI2010}};
  \end{align*}
  where $t,s$ are integers and $q'$ is a prime power.

In the general case, for arbitrary $r,R, q$, in particular when $r\ne tR$ and $q$ is an arbitrary prime power, the problem
of achieving the bound \eqref{eq1:lowbound} is open.

For $r=tR+1$, $R\ge2$, $t\ge1$, in \cite{BDGMP-R2R3CC_2019,DMP-R=3Redun2019,DMP-rad3arXiv2023,Nagy,DMP-AMC2021}, see also the references therein,
  upper bounds of the following forms are obtained:
\begin{align}\label{eq1:upbound}
  &  \ell_q(tR+1,R)\le cq^{(r-R)/R}\cdot\sqrt[R]{\ln q},~ q\text{ is an arbitrary prime power}, ~q>q_0,\db\\
  & \hspace{1cm}q_0\text{ is a fixed value that depends on the approach used to
obtain the bound};\notag\db\\
&\label{eq1:asympupbound}
\ell_q(tR+1,R)\le c^{knw}_R q^{(r-R)/R}\cdot\sqrt[R]{\ln q}+o(q^{(r-R)/R}),~ q\text{ is an arbitrary prime power},\db\\
&\hspace{1cm}q\text{ is large enough}.\notag
  \end{align}
In the bounds \eqref{eq1:upbound}, \eqref{eq1:asympupbound}, the ``price'' of the non-restricted structure of $q$ is the relatively small factor $\sqrt[R]{\ln q}$.
  The bound \eqref{eq1:asympupbound} is an \emph{asymptotic upper bound}.

For $R\ge3$,   the smallest known constants $c^{knw}_R$  are obtained in \cite{DMP-AMC2021,DMP-rad3arXiv2023} where we have
\begin{equation}\label{eq1:c knw}
  c^{knw}_R=\left\{\begin{array}{@{}cl@{}}
                     \sqrt[3]{18}\thickapprox2.6207& \text{if }  R=3,~r=3t+1 \text{ \cite{DMP-rad3arXiv2023}}\\
                    \frac{R}{R-1}\sqrt[R]{R(R-1)\cdot R!} & \text{if } R\ge3,~r=R+1 \text{ \cite[equations (3.1), (6.18)]{DMP-AMC2021}} \\
                    3.43R  & \text{if } R\ge3,~ r=tR+1, ~t\ge2 \text{ \cite[equation (3.13)]{DMP-AMC2021}}
                   \end{array}
  \right..
\end{equation}

In this paper, for $R\ge3$, we obtain new asymptotic upper bounds in the form \eqref{eq1:asympupbound}. We essentially decrease the known constants  before $q^{(r-R)/R}\cdot\sqrt[R]{\ln q}$. We denote
\begin{equation}\label{eq1:c new}
  c^{new}_{R}\triangleq\sqrt[R]{\frac{R!}{R^{R-2}}}.
\end{equation}

Lemma \ref{lem_lim_cnew} states that
\begin{equation}
 \lim_{R\rightarrow\infty} c^{new}_{R}=\lim_{R\rightarrow\infty}\sqrt[R]{\frac{R!}{R^{R-2}}}=\frac{1}{e}\thickapprox0.3679.
\end{equation}

Our main result is as follows, see Sections \ref{sec:est_size}, \ref{sec:qmconc}:
\begin{theorem}
Let $R\ge3$ be fixed. For the smallest
length of a $ q $-ary linear code with codimension (redundancy) $r=tR+1$ and covering radius $R$ (i.e. for the length function $\ell_q(tR+1,R)$) and for the smallest size $s_q(tR,R-1)$ of an $(R-1)$-saturating set in the projective space $\PG(tR,q)$
the following asymptotic upper bounds hold:
  \begin{align}\label{eq1:newbound}
   &\bullet~\ell_q(tR+1,R) =s_q(tR,R-1)\le
      c^{new}_{R}\cdot q^{(r-R)/R}\cdot\sqrt[R]{\ln q}+o(q^{(r-R)/R}),\db\\
    &\hspace{1cm} r=tR+1,~t\ge1,~q\text{ is an arbitrary prime power},~q \text{ is large enough}; \notag\\
    &\bullet~\text{ if additionally }R\text{ is large enough, then  in \eqref{eq1:newbound} we have }c^{new}_{R}\thicksim\frac{1}{e}\thickapprox0.3679. \label{eq1:newboundRbig}
  \end{align}
  The bounds are provided by infinite families of $[n,n-r]_qR$ codes of the corresponding lengths. For $t=1$ the codes have minimum distance $d=R+1$ and they are  AMDS codes.
\end{theorem}

By \eqref{eq1:c knw}--\eqref{eq1:newboundRbig}, for $q$ large enough, we have
\begin{equation}\label{eq1:knw/new}
\frac{c^{knw}_R}{c^{new}_{R}}
=\left\{\begin{array}{@{}cl@{}}
\sqrt[3]{18}/\sqrt[3]{3!/3^{3-2}}\thickapprox2.08&\text{if }R=3,~r=3t+1,~t\ge1 \\
\frac{R^2}{R-1}\sqrt[R]{\frac{R-1}{R}}\thickapprox R+1&\text{if }R\ge4,~r=R+1 \\
3.43R/\sqrt[R]{R!/R^{R-2}}&\text{if }R\ge4,~r=tR+1,~t\ge2\\
3.43eR\thickapprox 9.32 R&\text{if }R\text{ is large enough},~r=tR+1,~t\ge2
                                     \end{array}
\right..
\end{equation}
 By Section \ref{sec:propCnew} and tables in it, $c^{new}_{R}$ \eqref{eq1:c new}   is upper bounded by a decreasing function of~$R$. The ratio $c^{knw}_{R}/c^{new}_R$ \eqref{eq1:knw/new} is an increasing  function of $R$.   When $R$ increases from 4 to 150, then $c^{new}_{R}$ decreases from  $1.1067 \thickapprox0.2767R$ to $0.4024 \thickapprox 0.0027R $ ; $c^{knw}_{R}/c^{new}_R$ for $r=R+1$ increases from 4.9632 to 151; and $c^{knw}_{R}/c^{new}_R$ for $r=tR+1,~t\ge2$ (i.e. $3.43R/c^{new}_{R}$) increases from $12\thickapprox3.1R$ to $1279\thickapprox8.52R$. Moreover, if $r=tR+1,~t\ge2$, and $R$ is large enough, then $c^{knw}_{R}/c^{new}_R\thickapprox 9.32R$.

So, the new bounds are essentially better than the known ones.

We use a geometrical approach to the case $t=1$. We propose Construction B obtaining a relatively small  $(R-1)$-saturating $n$-set in $\PG(R,q)$ by a step-by-step algorithm. The set corresponds to an $[n,n-(R+1),R+1]_qR$ code. Note that, as the minimum distance $d = R+1$, the code is AMDS.
We estimate the code size that gives the upper bounds on $s_q(R,R-1)=\ell_q(R+1,R)$.

For $t\ge2$, we use a lift-construction for covering codes.  It is a variant of the so-called ``$q^m$-concatena\-ting constructions'' proposed in \cite{Dav90PIT} and developed in \cite{DGMP-AMC,DavOst-IEEE2001,DavOst-DESI2010}, see also the references therein and \cite[Section~5.4]{CHLS-bookCovCod}. The $q^m$-concatenating constructions obtain infinite families of covering codes with growing codimension using a starting code with a small one. The covering density of the codes from the infinite families is approximately the same as for the starting code.

We take the obtained $[n,n-(R+1),R+1]_qR$ code as the starting one for the $q^m$-concatena\-ting construction and obtain an infinite family of covering codes with growing codimension $r=tR+1$, $t\ge1$. The family provides the upper bound on $\ell_q(tR+1,R)$.

 The paper is organized as follows.  Section \ref{sec:constrB} describes Construction B that obtains $(R-1)$-saturating $n$-sets in $\PG(R,q)$ corresponding to $[n,n-(R+1),R+1]_qR$ AMDS codes. In Section \ref{sec:est_size}, we give estimates of sizes of saturating sets obtained by  Construction B and the corresponding  upper bounds.  In Section \ref{sec:qmconc}, asymptotic upper bounds  on the length function $\ell_q(tR+1,R)$ are obtained for growing $t\ge1$. The  bounds are provided by infinite families of covering codes with growing codimension $r=tR+1$, $t\ge1$, created by the $q^m$-concatenating construction. In Section \ref{sec:propCnew}, we investigate properties of the new bounds and show that they are essentially better than the known ones.

\section{New Construction B of $(R-1)$-saturating sets in $\PG(R,q)$, $R\ge3$}\label{sec:constrB}

In this section, for any $q$ and $R\ge3$, we propose a new Construction B of $(R-1)$-saturating sets in $\PG(R,q)$. It is an essential (non-obvious and non-trivial) modification of Construction A of~\cite{DMP-AMC2021}. For Construction B, the points of the $(R-1)$-saturating $n$-set in $\PG(R,q)$ (in homogeneous coordinates), treated as columns, form a parity check matrix of an $[n,n-(R+1),R+1]_qR$ code. The minimum distance $d=R+1$ is provided by Construction B. In Section~\ref{sec:qmconc}, this code is used as a starting one for lift-constructions obtaining infinite families of covering codes with growing codimension $r=tR+1$, $t\ge1$.

We construct an $(R-1)$-saturating set in $\PG(R,q)$  by a step-by-step iterative process adding $R$ new points to the current set in
every step.


\subsection{Notations and definitions}\label{subsec:Not&Def}

$\bullet$ We say that a point $P$ of $\PG(R,q)$ is \emph{$\rho$-covered} by a point set $\K\subset\PG(R,q)$ if $P$ lies in a $\rho$-subspace generated by $\rho+1$ points of $\K$ in general positions.  In this case,  the set~$\K$ \emph{$\rho$-covers} the point $P$.
If $\rho$ is clear by the context, one can say simply ``covered" and ``covers" (resp. ``uncovered'' and ``does not cover'').

We denote by $\dim(H)$ the dimension of a subspace $H$.
Let $V_1$ and $V_2$ be some subspaces of $\PG(N,q)$. Clearly, $\dim(V_1\cup V_2)\le N$. By Grassman formula, we have
\begin{equation*}
    \dim(V_1\cap V_2)=\dim(V_1)+\dim(V_2)-\dim(V_1\cup V_2).
\end{equation*}
This relation is used when we consider intersections of subspaces.

Let $\theta_{N,q}=(q^{N+1}-1)/(q-1)$ be the number of points in the projective space $\PG(N,q)$.

Let $A_u$ be a point of $\PG(R,q)$, $u=1,\ldots,\theta_{R,q}$. Point numbers are not fixed before the beginning of the iterative process. Points are numbered as they are included in the saturating set that we are building.

$\bullet$ For the iterative process, let $w\ge0$ be the step number.
Let
\begin{equation}\label{eq31:Pc0}
 \Pc_{0,R}\triangleq\{A_1,\ldots,A_{R}\}\subset\PG(R,q)
\end{equation}
be a \emph{starting} $R$-set  such that all its points are in general position. For example,
we can take $R$ arbitrary points of any arc in $\PG(R,q)$ as  $\Pc_{0,R}$. Recall that in $\PG(R,q)$, an arc is a set of points no $R + 1$ of which belong to the same hyperplane. Any $R + 1$ points of an arc are in general position. In particular, we can take the points in homogeneous coordinates
$
  A_1=(\underbrace{1,0,\ldots,0}_{R+1}), A_2=(\underbrace{0,1,0,\ldots,0}_{R+1}),\ldots, A_R=(\underbrace{0,\ldots,0,1,0}_{R+1}).
$\\

Let $\K_w$ be the current $(w+1)R$-set obtained after the
$w$-th step of the process. We put $\K_0=\Pc_{0,R}$. We have
\begin{equation*}
  \#\K_{w-1}=wR.
\end{equation*}
 In the $w$-th step, let
\begin{equation}\label{eq31:PcwR}
  \Pc_{w,R}\triangleq\{A_{wR+1},A_{wR+2},\ldots,A_{wR+R}\},~w\ge1,
\end{equation}
be an $R$-set  of points that are added to the current set $\K_{w-1}$ to obtain the next set $\K_{w}$;
\begin{equation}\label{eq31:Kw}
 \K_{w}=\K_{w-1}\cup\Pc_{w,R}=\Pc_{0,R}\cup\Pc_{1,R}\cup\ldots\cup\Pc_{w,R},~\#\K_{w}=(w+1)R,~w\ge1.
\end{equation}
We denote
\begin{equation}\label{eq31:Pcwi}
\Pc_{w,i}\triangleq\{A_{wR+1},A_{wR+2},\ldots,A_{wR+i}\} \subseteq\Pc_{w,R},~ i=1,2,\ldots,R,~ w\ge1;~\Pc_{w,0}\triangleq\emptyset.
\end{equation}
 We call $A_{wR+1}$ \emph{the $w$-th leading point}.

The $w$-step of the process consists of $R$ sub-steps; on the $i$-th sub-step we add to the current set a point $A_{wR+i}$. Thus, after the
$i$-th sub-step of the $w$-step we obtain the current set in the form
\begin{equation*}
  \K_{w-1}\cup\{A_{wR+1}\}\cup\ldots\cup\{A_{wR+i}\}=\K_{w-1}\cup\Pc_{w,i},~ i=1,2,\ldots,R.
\end{equation*}

To build $\Pc_{w,1},\ldots,\Pc_{w,R}$, we take a hyperplane of $\PG(R,q)$, say $\Pi_{w}$, skew to $\K_{w-1}$, i.e.
\begin{equation}\label{eq31:Piw}
 \Pi_{w}\subset\PG(R,q),~ \Pi_{w}\cap\K_{w-1}=\emptyset,~\dim(\Pi_{w})=R-1,~\#\Pi_{w}=\theta_{R-1,q}.
\end{equation}
In $\PG(R,q)$, a blocking set regarding hyperplanes contains at least $\theta_{1,q}$
points \cite{BoseBurt}. Therefore the saturating sets with the sizes proved in this paper cannot be a blocking set regarding hyperplanes. So, the needed hyperplane $\Pi_{w}$ exists.

We choose all the points  of $\Pc_{w,i}$, $i=1,2,\ldots,R$, from $\Pi_{w}$, i.e.
\begin{equation}\label{eq31:Pcwi subs Piw}
\Pc_{w,i}\subseteq \Pc_{w,R}\subset\Pi_w,~i=1,2,\ldots,R,~w\ge1.
\end{equation}

$\bullet$ We denote
\begin{equation}\label{eq31:mathfracBwi}
\BB_{w,i}\triangleq\binom{\#\K_{w-1}+i-1}{R-1}=\binom{wR+i-1}{R-1},~i\ge1.
\end{equation}
Here $\#\K_{w-1}+i-1=wR+i-1$ is the size of the current set $\K_{w-1}\cup\Pc_{w,i-1}$ obtained after the $(i-1)$-th sub-step of the $w$-th step of the process;
remind that $\Pc_{w,0}\triangleq\emptyset$, see~\eqref{eq31:Pcwi}. This implies that $\BB_{w,i}$ is the number of distinct $(R-1)$-subsets of the current set $\K_{w-1}\cup\Pc_{w,i-1}$.

For the given $i\in\{1,2,\ldots, R\}$, we consider $\BB_{w,i}$ \emph{distinct} $(R-1)$-subsets consisting of \emph{distinct points} of $\K_{w-1}\cup\Pc_{w,i-1}$. We denote such a subset by
$\D_{w,i}^j$ with
 \begin{equation}\label{eq31:Dwij}
 \D_{w,i}^j\subset\K_{w-1}\cup\Pc_{w,i-1},~\#\D_{w,i}^j=R-1,~j=1,\ldots,\BB_{w,i},~\D_{w,i}^u\ne\D_{w,i}^v \text{ if }u\ne v.
 \end{equation}
 All the points of $\D_{w,i}^j$ generate a subspace, say $ \V_{w,i}^j$, which meets $\Pi_{w}$ in a subspace, say $\T_{w,i}^j$.
In other words,
\begin{equation}\label{eq31:Vwij Twij}
  \V_{w,i}^j\triangleq\langle\D_{w,i}^j\rangle,~1\le\dim(\V_{w,i}^j)\le R-2, ~\T_{w,i}^j\triangleq\V_{w,i}^j\cap\Pi_{w}.
\end{equation}
The subspace $ \V_{w,i}^j$ has maximal possible dimension $R-2$ if and only if all the points of $\D_{w,i}^j$ are in general position.

  By \eqref{eq31:PcwR}--\eqref{eq31:Vwij Twij}, if in \eqref{eq31:Dwij} $i= R$ then $i-1= R-1,~\Pc_{w,i-1}=\Pc_{w,R-1}$, and one and only one $(R-1)$-subset $\D_{w,R}^j$ consists entirely of new points $A_{wR+1},A_{wR+2},\ldots,A_{wR+R-1}$. For definiteness we put here $j=1$. As all the new points lie in $\Pi_w$, see \eqref{eq31:Pcwi subs Piw}, we have
  \begin{equation}\label{eq31:i=R}
   \D_{w,R}^1\triangleq\Pc_{w,R-1}\subset\Pi_w,~\T_{w,R}^1=\V_{w,R}^1\subset\Pi_w,~\dim(\T_{w,R}^1)=\dim(\V_{w,R}^1)\le R-2.
  \end{equation}
  If $1\le i\le R-1,~j=1,\ldots,\BB_{w,i},$ or $i=R,~j=2,\ldots,\BB_{w,R},$  then
  \begin{align}\label{eq31:i<=R-1}
  \D_{w,i}^j\ne\Pc_{w,i-1};~  \D_{w,i}^j, \V_{w,i}^j\not\subset\Pi_w;~\T_{w,i}^j\subset\V_{w,i}^j;~ \dim(\T_{w,i}^j)=\dim(\V_{w,i}^j)-1\le R-3.
  \end{align}

 We denote
 \begin{equation}\label{eq31:mathfrak{T}}
 \mathfrak{T}_{w,i}\triangleq\bigcup_{j=1}^{\BB_{w,i}}\T_{w,i}^j,~i=1,2,\ldots, R;~\mathfrak{T}_{w,0}\triangleq\emptyset.
 \end{equation}
By \eqref{eq31:Vwij Twij}, the points of $\T_{w,i}^j$ are not in general position with the points of $\D_{w,i}^j$. Therefore, the set $\mathfrak{T}_{w,i}$ contains all the points $Q_k$ of $\Pi_w$ such that there exists at least one set of $R-1$ points of $\K_{w-1}\cup\Pc_{w,i-1}$ which are not in general position with $Q_k$. Thus, the next new point $A_{wR+i}$ cannot be taken from $\mathfrak{T}_{w,i}$. If we remove $\mathfrak{T}_{w,i}$ from $\Pi_{w}$ then we can take $A_{wR+i}$ from the obtained subset. We denote
 \begin{equation}\label{eq31:Piwi}
 \Pi_{w,i}\triangleq\Pi_{w}\setminus\mathfrak{T}_{w,i},~i=1,2,\ldots, R;~\Pi_{w,0}\triangleq\Pi_{w}.
 \end{equation}

We are going to take $A_{wR+i}$ from $\Pi_{w,i}$; this implies, see \eqref{eq31:mathfrak{T}}, \eqref{eq31:Piwi}, that all the points of $\Pc_{w,i}$, $i=1,2,\ldots,R$, belong to $\Pi_{w,1}$, i.e.
\begin{equation}\label{eq31:Pcw+1 subs Piw+1}
\Pc_{w,i}\subseteq \Pc_{w,R}\subset\Pi_{w,1}\subset\Pi_w,~i=1,2,\ldots,R,~w\ge1.
\end{equation}
This will allow us to provide that $\Pc_{w,R}$ covers all the points of $\Pi_{w}$, see Section \ref{subsec:PcwR} for details.

$\bullet$ Let $\U_w$ be the subset of $\PG(R,q)\setminus\K_w$ consisting of the points that are \emph{not $(R-1)$-covered} by~$\K_w$, $w\ge0$.
 Let $\Delta_{w}(\Pc_{w,R})$ be the number of new covered points in $\U_{w-1}$ after adding $\Pc_{w,R}$ to $\K_{w-1}$, $w\ge1$, i.e.
 \begin{equation}\label{eq31:Delta_def}
   \Delta_{w}(\Pc_{w,R})\triangleq\#\U_{w-1}-\#\U_{w}, ~w\ge1.
 \end{equation}

We denote $\delta_{w}(A_{wR+1})$ the number of new covered points in $\U_{w-1}\setminus\Pi_w$ after adding the $w$-th leading point $A_{wR+1}$ to $\K_{w-1}$. By \eqref{eq31:Pcw+1 subs Piw+1}, the points of $\Pi_{w,1}$ are candidates to be $A_{wR+1}$, see Sections \ref{subsec:PcwR} and \ref{subsec:leading} for details. Let $\mathbb{S}_w$ be the sum of the number of new covered points in $\U_{w-1}\setminus\Pi_w$ over all points $P$ of $\Pi_{w,1}$, i.e.
 \begin{equation}\label{eq31:Sw}
   \mathbb{S}_w\triangleq\sum_{P\in\Pi_{w,1}}\delta_{w}(P).
 \end{equation}


\subsection{The construction of the $R$-set $\Pc_{w,R}$}\label{subsec:PcwR}
\begin{lemma}\label{lem32:base}
  Let $i\in\{1,2,\ldots, R\}$, $w\ge1$.  Let any $R$ points of $\K_{w-1}\cup\Pc_{w,i-1}$ are in general position. Then the following holds:
  \begin{description}
    \item[(i)] All the points of $\D_{w,i}^j$, $j=1,\ldots,\BB_{w,i}$, are in general position.
    \item[(ii)]
The subspace $ \V_{w,i}^j$ has maximal possible dimension $R-2$, i.e.
    \begin{align}\label{eq32:Vwij}
& \V_{w,i}^j\triangleq\langle\D_{w,i}^j\rangle,~\dim(\V_{w,i}^j)=R-2,~\#\V_{w,i}^j=\theta_{R-2,q},~j=1,\ldots,\BB_{w,i}.
\end{align}
All the subspaces $ \V_{w,i}^j$ are distinct, i.e.
\begin{equation}\label{eq32:Vmij distinct}
 \V_{w,i}^u\ne\V_{w,i}^v\text{ if }u\ne v.
\end{equation}

The subspace $\T_{w,i}^j$ entirely lies in the hyperplane $\Pi_w$ and has  dimension $R-2$, if $i=R,~j=1$; otherwise $\T_{w,i}^j$ intersects $\Pi_w$ and has  dimension $R-3$. In other words,
\begin{align}
&\T_{w,R}^1=\V_{w,R}^1\subset\Pi_w,~\dim(\T_{w,R}^1)=R-2,~\#\T_{w,R}^1=\theta_{R-2,q};\label{eq32:Twij b}\db\\
&\T_{w,i}^j\triangleq\V_{w,i}^j\cap\Pi_{w},~\dim(\T_{w,i}^j)=R-3,~\#\T_{w,i}^j=\theta_{R-3,q},\label{eq32:Twij a}\db\\
&\text{if }1\le i\le R-1,~j=1,\ldots,\BB_{w,i},\text{ or }i= R,~j=2,\ldots,\BB_{w,R}.\notag
\end{align}

The union of the subspaces $\T_{w,i}^j$ with a fixed $i$, i.e. $\mathfrak{T}_{w,i}\triangleq\bigcup_{j=1}^{\BB_{w,i}}\T_{w,i}^j$, has the size lying in the following regions:
\begin{align}
&\theta_{R-3,q}\le\#\mathfrak{T}_{w,i}\le\BB_{w,i}\theta_{R-3,q}\text{ if }1\le i\le R-1;\label{eq32:sizemathfracT a}\db\\
&\theta_{R-2,q}\le\#\mathfrak{T}_{w,R}\le\BB_{w,R}\theta_{R-3,q}+q^{R-2}.\label{eq32:sizemathfracT b}
\end{align}
The subset of the hyperplane $\Pi_w$ obtained by moving of $\mathfrak{T}_{w,i}$, i.e. $\Pi_{w,i}\triangleq\Pi_{w}\setminus\mathfrak{T}_{w,i}$,  has the size lying in the following regions:
\begin{align}
&q^{R-2}(q+1)\ge\#\Pi_{w,i}\ge\theta_{R-3,q}\left(\frac{q^{R}-1}{q^{R-2}-1}-\BB_{w,i}\right)\text{ if }1\le i\le R-1;\label{eq32:sizePwi a}\db\\
&q^{R-1}\ge\#\Pi_{w,R}\ge\theta_{R-3,q}\left(\frac{q^{R}-1}{q^{R-2}-1}-\BB_{w,R}\right)-q^{R-2}\notag\db\\
&\label{eq32:sizePwi b}=q^{R-3}\left(\frac{q^{R}-1}{q^{R-2}-1}-\BB_{w,R}-q\right)+
\left\{\begin{array}{@{}c@{}c@{}}0
&\text{ if }R=3\\
\theta_{R-4,q}\left(\frac{q^{R}-1}{q^{R-2}-1}-\BB_{w,R}\right)&\text{ if }R\ge4
\end{array}
\right..
\end{align}
We have the following embeddings:
\begin{align}
&\mathfrak{T}_{w,i}\subseteq\mathfrak{T}_{w,i+1},~i\neq R;~\Pi_{w,i}\subseteq\Pi_{w,i-1}.\label{eq32:subset}
\end{align}
    \item[(iii)] Let $\#\Pi_{w,i}\ge1$. Any point of $\Pi_{w,i}=\Pi_w\setminus\mathfrak{T}_{w,i}$ is in general position with any $R-1$ points of $\K_{w-1}\cup\Pc_{w,i-1}$.
    \item[(iv)]   Let $\#\Pi_{w,i}\ge1$. Let $A_{wR+i}$ be any point of $\Pi_{w,i}$. Then any $R$ points of $\K_{w-1}\cup\Pc_{w,i-1}\cup \{A_{wR+i}\}=\K_{w-1}\cup\Pc_{w,i}$ (and, in particular, of $\K_{w-1}\cup\Pc_{w,R-1}\cup \{A_{wR+R}\}=\K_{w-1}\cup\Pc_{w,R}=\K_w$) are in general position.
  \end{description}
\end{lemma}

\begin{proof}
\begin{description}
  \item[(i)]  The assertion follows from the  hypothesis and from \eqref{eq31:mathfracBwi}, \eqref{eq31:Dwij}.

  \item[(ii)]  The assertions follow from the case (i) and from the constructions \eqref{eq31:mathfracBwi}--\eqref{eq31:Piwi}. In \eqref{eq32:sizePwi a} and \eqref{eq32:sizePwi b} we do simple transformations. About \eqref{eq32:Vmij distinct} we note that if $u\ne v$ but $\V_{w,i}^u=\V_{w,i}^v$, then $\D_{w,i}^u\ne\D_{w,i}^v$, $\#(\D_{w,i}^u\cup\D_{w,i}^v)\ge R$, $\langle\D_{w,i}^u\rangle=\langle\D_{w,i}^v\rangle$. Therefore, the set $\D_{w,i}^u\cup\D_{w,i}^v$, belonging to $\K_{w-1}\cup\Pc_{w,i-1}$, lies in the $(R-2)$-subspace $\langle\D_{w,i}^u\rangle$ and any $R$ of its points are not in general positions; contradiction with hypothesis.

  \item[(iii)]  By the construction \eqref{eq31:Vwij Twij}, the points of $\V_{w,i}^j$ and $\T_{w,i}^j$ are not in general position with the points of $\D_{w,i}^j$. By \eqref{eq31:Vwij Twij}, \eqref{eq31:mathfrak{T}}, the set $\mathfrak{T}_{w,i}$ contains all the points of $\Pi_w$ such that for every point there exist at least one set of  $R-1$ points of $\K_{w-1}\cup\Pc_{w,i-1}$
       which are not in general position with it. The assertion follows.
    \item[(iv)] It follows from the case (iii).
\end{description}
\end{proof}

  \begin{corollary}\label{cor32:a}
  Let $w\ge1$. Let $\#\Pi_{w,i}\ge1$, $i=1,2,\ldots, R$.  Assume that we form the $R$-set $\Pc_{w,R}$ of \eqref{eq31:PcwR}, obtaining sequentially the sets $\Pc_{w,1}, \Pc_{w,2},\ldots,\Pc_{w,R}$ of
    \eqref{eq31:Pcwi} so that the point $A_{wR+i}$ is taken from $\Pi_{w,i}$ of~\eqref{eq31:Piwi}, i.e.
 \begin{equation}\label{eq31:sequen}
   A_{wR+i}\in\Pi_{w,i},~w\ge1,~i=1,2,\ldots, R.
 \end{equation}
 Then for all $w\ge1$ and all $i=1,2,\ldots, R+1$, any $R$ points of $\K_{w-1}\cup\Pc_{w,i-1}$ (in particular, of $\K_{w-1}\cup\Pc_{w,R}=\K_w$) are in general position.
 \end{corollary}

 \begin{proof}
  By \eqref{eq31:Pc0} and \eqref{eq31:Pcwi}, for $w=1$, $i=1$,  any $R$ points of $\K_{w-1}\cup\Pc_{w,i-1}=\K_{0}\cup\Pc_{1,0}=\K_{0}\cup\emptyset=\Pc_{0,R}$   are in general position.
 Then the assertions can be proved by induction on $w$ and $i$ by using Lemma~\ref{lem32:base}.
 \end{proof}

 \begin{remark}\label{rem32}
   To provide the assertions of  Corollary \ref{cor32:a}, it is sufficient to take any point of $\Pi_{w,1}$ as the $w$-th leading point $A_{wR+1}$. However, to obtain a small saturating set, the leading point should have also some other additional properties, see Section \ref{subsec:leading}.
 \end{remark}

 \begin{corollary}\label{cor32:b}
   Under the conditions of Corollary \emph{\ref{cor32:a}}, the $R$-set $\Pc_{w,R}$ covers all the points of the hyperplane $\Pi_{w}$.
 \end{corollary}

 \begin{proof}
   We use Corollary \ref{cor32:a} and \eqref{eq31:Pcw+1 subs Piw+1}.
 \end{proof}

\begin{corollary}\label{cor32:c}
  To guarantee the condition $\#\Pi_{w,i}\ge1$, it is sufficient that
  \begin{align}\label{eq32:suffic}
   \frac{q^{R}-1}{q^{R-2}-1}>\left\{\begin{array}{lcl}
                                      \BB_{w,i} & \text{if} & 1\le i\le R-1 \\
                                     \BB_{w,R}+q & \text{if} & i=R
                                    \end{array}\right..
  \end{align}
\end{corollary}

\begin{proof}
  The assertion follows from \eqref{eq32:sizePwi a} and \eqref{eq32:sizePwi b}.
\end{proof}


\subsection{The choice of the $w$-th leading point $A_{wR+1}\in\Pi_{w,1}$}\label{subsec:leading}
Assume that $i=1$, $w\ge1$, and any $R$ points of $\K_{w-1}\cup\Pc_{w,1-1}=\K_{w-1}\cup\emptyset=\K_{w-1}$ are in general position. Let $\#\Pi_{w,1}\ge1$, see Corollary \ref{cor32:c}.

 To investigate the sum of the number of new covered points in $\U_{w-1}\setminus\Pi_w$ over all points $P$ of $\Pi_{w,1}$, $\mathbb{S}_w$ \eqref{eq31:Sw}, \emph{we fix a point $B\in\U_{w-1}\setminus\Pi_w$}, i.e. $B$ is not covered by $\K_{w-1}$ and $B\notin\Pi_w$.

 Let $j=1,\ldots,\BB_{w,1}$. By Lemma \ref{lem32:base}(i), all the points of the $R$-set $\D_{w,1}^j\cup\{B\}$  are in general position, otherwise $B$ would be covered by $\K_{w-1}$. The points of $\D_{w,1}^j\cup\{B\}$  define a hyperplane, say $\Sigma_{w,B}^j$, with
  \begin{equation}\label{eq33:Sigma}
   \Sigma_{w,B}^j\triangleq\langle\D_{w,1}^j\cup\{B\}\rangle\subset\PG(R,q),~\dim(\Sigma_{w,B}^j)=R-1,~\#\Sigma_{w,B}^j=\theta_{R-1,q}.
  \end{equation}
 As $B\notin\Pi_w$, we have $\Sigma_{w,B}^j\ne\Pi_w$. So, the hyperplanes $\Sigma_{w,B}^j$ and $\Pi_w$ intersect. The intersection of $\Sigma_{w,B}^j$ and $\Pi_w$ is an $(R-2)$-subspace, say $\Gamma_{w,B}^j$, such that
  \begin{equation}\label{eq33:Gamma}
    \Gamma_{w,B}^j\triangleq\Sigma_{w,B}^j\cap\Pi_w,~\dim(\Gamma_{w,B}^j)=R-2,~\#\Gamma_{w,B}^j=\theta_{R-2,q}.
  \end{equation}

  By \eqref{eq31:Vwij Twij}, \eqref{eq32:Vwij}, \eqref{eq33:Sigma}, the $(R-2)$-subspace $\V_{w,1}^j$ lies in $\Sigma_{w,B}^j$. As the $(R-2)$-subspaces $\V_{w,1}^j$ and $\Gamma_{w,B}^j$ lie in  the same hyperplane $\Sigma_{w,B}^j$ they meet in some $(R-3)$-subspace. This $(R-3)$-subspace is $\T_{w,1}^j$ since $\Gamma_{w,B}^j\subset\Pi_w$, see \eqref{eq33:Gamma}, and $\T_{w,1}^j=\V_{w,1}^j\cap\Pi_w\subset\Pi_w$, see \eqref{eq32:Twij a}.
 Thus,
\begin{equation}\label{eq33:Tw1j}
  \T_{w,1}^j=\V_{w,1}^j\cap\Pi_w=\V_{w,1}^j\cap\Gamma_{w,B}^j.
\end{equation}

The average value $ \delta_{w}^\text{aver}$ of the number of new covered points in $\U_{w-1}\setminus\Pi_w$ after adding $P$ to $\K_{w-1}$, $\delta_{w}(P)$, over all points of $\Pi_{w,1}$ is (see \eqref{eq31:Sw}, \eqref{eq32:sizePwi a})
 \begin{equation}\label{eq33:aver}
    \delta_{w}^\text{aver}\triangleq\frac{\mathbb{S}_w}{\#\Pi_{w,1}}=
    \frac{\sum\limits_{P\in\Pi_{w,1}}\delta_{w}(P)}{\#\Pi_{w,1}}\ge
    \frac{\sum\limits_{P\in\Pi_{w,1}}\delta_{w}(P)}{q^{R-2}(q+1)}.
 \end{equation}
Obviously, there exists a non-empty subset of $\Pi_{w,1}$  such that for every its point $P$ we have $\delta_{w}(P)\ge\delta_{w}^\text{aver}$.
 Any point from this subset can be taken as the $w$-th leading point $A_{wR+1}\in\Pi_{w,1}$. So,
 \begin{equation}\label{eq33:delta Alead}
  \delta_{w}(A_{wR+1})\ge\frac{\sum\limits_{P\in\Pi_{w,1}}\delta_{w}(P)}{\#\Pi_{w,1}}\ge
    \frac{\sum\limits_{P\in\Pi_{w,1}}\delta_{w}(P)}{q^{R-2}(q+1)}.
 \end{equation}


\subsection{Stages of the iterative process of Construction B}\label{subsec:iterproc}
 The iterative process for the construction of an $(R-1)$-saturating set is as follows:
\begin{itemize}
  \item We assign the starting set $\Pc_0$ in accordance with \eqref{eq31:Pc0} and put $w=0$, $\K_0=\Pc_0$.
  \item In every $w$-th step, $w\ge1$, we do the following:
  \begin{description}
    \item[$\bullet\bullet$]
    Choose a hyperplane $\Pi_w$ of $\PG(R,q)$ skew to $\K_{w-1}$, see \eqref{eq31:Piw}.

\item[$\bullet\bullet$]
Sequentially execute  $R$ sub-steps with $i=1,2,\ldots,R$.
    \item[$\bullet\bullet\bullet$]
    Put $i=1$. Form the sets $\D_{w,1}^j$, the subspaces $ \V_{w,1}^j$,  $\T_{w,1}^j$, the set $ \mathfrak{T}_{w,1}$, and the subset $\Pi_{w,1}\subset\Pi_w$,
  in accordance with \eqref{eq31:Dwij}, \eqref{eq31:Vwij Twij}, \eqref{eq31:i<=R-1}--\eqref{eq31:Piwi}. Choose the $w$-th leading point $A_{wR+1}$  from  $\Pi_{w,1}$ in accordance with Section \ref{subsec:leading}. This provides that any $R$  points from $\K_{w-1}\cup\{A_{wR+1}\}$ are in general position, see Corollary~\ref{cor32:a}.
    Also, this gives a base to obtain  a saturating set of small size, see \eqref{eq33:aver}, \eqref{eq33:delta Alead}.

      \item[$\bullet\bullet\bullet$]   For $i=2,3,\ldots, R$, sequentially form the sets $\D_{w,i}^j$, the subspaces $ \V_{w,i}^j$,  $\T_{w,i}^j$, the set $ \mathfrak{T}_{w,i}$, and the subset $\Pi_{w,i}\subset\Pi_w$,
  in accordance with \eqref{eq31:Dwij}--\eqref{eq31:Piwi}, and choose the point $A_{wR+i}$  from  $\Pi_{w,i}$ in accordance with Corollary \ref{cor32:a}. This forms the $R$-set  $\Pc_{w,R}\subset\Pi_w$ and the new current set $\K_w=\K_{w-1}\cup\Pc_{w,R}$ any $R$  points of which are in general position, see Corollary \ref{cor32:a},  Also, this provides that $\Pc_{w,R}$ covers all points of $\Pi_w$, see Corollary \ref{cor32:b}.

\item[$\bullet\bullet$] Count (or make an estimate of) the values of $\Delta_{w}(\Pc_{w,R})$ and $\#\U_{w}$, see \eqref{eq31:Delta_def} and Section \ref{sec:est_size}.
  \end{description}

  \item The process ends when $ \#\U_{w}\le R$. Finally, in the last $(w+1)$-st step, we add to $\K_w$ at most $R$ uncovered points to obtain an $(R-1)$-saturating set.
\end{itemize}

Note that in Construction B proposed in this paper, in every $w$-th step we provide that any $R$  points of  the current set $\K_w $ are in general position , whereas in Construction~A of \cite{DMP-AMC2021} only the points of the starting set are guaranteed to have this property.
This explains the improvement of the upper bounds obtained in this paper.

\begin{example}
  Let $R=3$. Then, see \eqref{eq31:Pc0},
 $\Pc_{0,3}\triangleq\{A_1,A_2,A_3\}\subset\PG(3,q);~
A_1=(1,0,0,0), A_2=(0,1,0,0), A_3=(0,0,1,0).$ Also, $\K_0=\Pc_{0,3}$ and $\#\K_{w-1}=3w$.
By \eqref{eq31:PcwR}--\eqref{eq31:Pcwi}, $\Pc_{w,1}=\{A_{3w+1}\}$, $\Pc_{w,2}=\{A_{3w+1},\A_{3w+2}\}$, $\Pc_{w,3}=\{A_{3w+1},\A_{3w+2},A_{3w+3}\}$ where $A_{3w+1}$ is the $w$-th leading point.

The $w$-step of the process consists of $3$ sub-steps; on the $i$-th sub-step we add to the current set a point $A_{3w+i}$. After the
$i$-th sub-step of the $w$-step we obtain the current set in the form
$ \K_{w-1}\cup\{A_{3w+1}\}\cup\ldots\cup\{A_{3w+i}\}=\K_{w-1}\cup\Pc_{w,i},~ i=1,2,3.$

The hyperplane $\Pi_{w}$, skew to $\K_{w-1}$, is a plane of size $q^2+q+1$;
$\BB_{w,i}=\binom{3w+i-1}{2}$ is the number of distinct pairs points of the current set $\K_{w-1}\cup\Pc_{w,i-1}$, $i=1,2,3.$  Such pairs are denoted by $\D_{w,i}^j$, $j=1,\ldots,\BB_{w,i}$; they generate $\BB_{w,i}$ distinct lines $\V_{w,i}^j$, see~\eqref{eq32:Vmij distinct}.

Let $i=1$. By construction, any $3$ points of $\K_{w-1}\cup\Pc_{w,1-1}=\K_{w-1}\cup\emptyset=\K_{w-1}$ are in general position. Every line $\V_{w,1}^j$ meets the plane $\Pi_{w}$ in the point $\T_{w,1}^j$. We remove the union $\mathfrak{T}_{w,1}=\bigcup_{j=1}^{\BB_{w,1}}\T_{w,1}^j$ of these points from the plane $\Pi_{w}$  and obtain $\Pi_{w,1}=\Pi_{w}\setminus\mathfrak{T}_{w,1}$. Then we choose the $w$-th leading point $A_{3w+1}$ from  $\Pi_{w,1}$ in accordance with Section \ref{subsec:leading}.
The choice of $A_{3w+1}$ from  $\Pi_{w,1}$ provides that any $3$  points from $\K_{w-1}\cup\{A_{3w+1}\}$ are in general position, see Corollary~\ref{cor32:a}.
 Moreover, we take a point of $\Pi_{w,1}$  providing \eqref{eq33:aver}, \eqref{eq33:delta Alead}, that gives a base to obtain  a saturating set of small size.

 Let $i=2$. By above, any $3$ points of $\K_{w-1}\cup\Pc_{w,2-1}=\K_{w-1}\cup\{A_{3w+1}\}$ are in general position. Every line $\V_{w,2}^j$ meets  $\Pi_{w}$ in the point $\T_{w,2}^j$. We remove the union $\mathfrak{T}_{w,2}=\bigcup_{j=1}^{\BB_{w,2}}\T_{w,2}^j$ from  $\Pi_{w}$  and obtain $\Pi_{w,2}=\Pi_{w}\setminus\mathfrak{T}_{w,2}$. Then we take any point from  $\Pi_{w,2}$ as $A_{3w+2}$ that provides that any $3$  points from $\K_{w-1}\cup\{A_{3w+1},A_{3w+2}\}$ are in general position, see Corollary~\ref{cor32:a}.

 Let $i=3$. By above, any $3$ points of $\K_{w-1}\cup\Pc_{w,3-1}=\K_{w-1}\cup\{A_{3w+1},A_{3w+2}\}$ are in general position. The 2-set $\D_{w,3}^1$ and the line $\V_{w,3}^1=\langle\D_{w,3}^1\rangle$ lie  in the plane $\Pi_{w}$ that implies $\T_{w,3}^1=\V_{w,3}^1\subset\Pi_w$. For $j\ge2$, every line $\V_{w,3}^j$ meets  $\Pi_{w}$ in the point $\T_{w,3}^j$. We remove the union $\mathfrak{T}_{w,3}=\bigcup_{j=1}^{\BB_{w,3}}\T_{w,3}^j$ from  $\Pi_{w}$  and obtain $\Pi_{w,3}=\Pi_{w}\setminus\mathfrak{T}_{w,3}$. Then we take any point from  $\Pi_{w,3}$ as $A_{3w+3}$ that provides that any $3$  points from $\K_{w-1}\cup\{A_{3w+1},A_{3w+2},A_{3w+3}\}$ are in general position, see Corollary~\ref{cor32:a}.
 \end{example}


\section{Estimates of sizes of the saturating sets obtained by  Construction~B. Upper bound on the length function $\ell_q(R+1,R)$}\label{sec:est_size}

\subsection{Estimates of the size of $\delta_{w}(A_{wR+1})$}\label{subsec:size leader}
We estimate the number of new covered points in $\U_{w-1}\setminus\Pi_w$ after adding the $w$-th leading point $A_{wR+1}$ to $\K_{w-1}$, $\delta_{w}(A_{wR+1})$, counting how many times an uncovered point is covered when adding $A_{wR+1}$ to $\K_{w-1}$.

We consider the relations \eqref{eq33:Sigma}--\eqref{eq33:Tw1j} and a point $B\in\U_{w-1}\setminus\Pi_w$, see Section \ref{subsec:leading}. Recall that the points of $ \T_{w,1}^j$ are not in general position with the points of $\D_{w,1}^j$. We denote
 \begin{equation}\label{eq41:wideGamma}
 \widehat{\Gamma}_{w,B}^j=\Gamma_{w,B}^j\setminus \T_{w,1}^j.
 \end{equation}
 Every point of $\widehat{\Gamma}_{w,B}^j$ is in general position with the points of $\D_{w,1}^j$; also,
 \begin{equation*}
   \#\widehat{\Gamma}_{w,B}^j=\theta_{R-2,q}-\theta_{R-3,q}=q^{R-2}.
 \end{equation*}
 By construction, the $q^{R-2}$-set $\widehat{\Gamma}_{w,B}^j$ is the affine point set of the $(R-2)$-subspace $\Gamma_{w,B}^j$.

   Thus, the hyperplane
  $\Sigma_{w,B}^j=\langle\D_{w,1}^j\cup\{B\}\rangle$ is generated  $q^{R-2}$ times when we  add in sequence all the points of $\Pi_{w,1}\subset\Pi_w$ to $\K_{w-1}$ for the calculation of $\mathbb{S}_w$ \eqref{eq31:Sw}.

The above is true for all $\BB_{w,1}$ sets $\D_{w,1}^j$. Moreover, consider the sets $\D_{w,1}^u$ and $\D_{w,1}^v$ with $u\ne v$. By the definition \eqref{eq31:Dwij}, $\D_{w,1}^u\ne\D_{w,1}^v$. The points of $\D_{w,1}^u\cup\{B\}$ (resp. $\D_{w,1}^v\cup\{B\}$) define a hyperplane $\Sigma_{w,B}^u$ (resp. $\Sigma_{w,B}^v$). No points of $\D_{w,1}^v\setminus\D_{w,1}^u$ lie in $\Sigma_{w,B}^u$, otherwise $B$ would be $(R-1)$-covered by~$\K_{w-1}$. So, the hyperplanes
 $\Sigma_{w,B}^u$ and $\Sigma_{w,B}^v$ are distinct.  If the corresponding $(R-2)$-subspaces $\Gamma_{w,B}^u=\Sigma_{w,B}^u\cap\Pi_w$ and $\Gamma_{w,B}^v=\Sigma_{w,B}^v\cap\Pi_w$ coincide with each other, then $\Sigma_{w,B}^u$ and $\Sigma_{w,B}^v$ have no common points outside $\Pi_w$, contradiction as $B\notin\Pi_w$. Thus, $\Gamma_{w,B}^u\ne\Gamma_{w,B}^v$.

We have proved that in $\Pi_{w,1}\subset\Pi_w$ we have $\BB_{w,1}$ distinct $(R-2)$-subspaces $\Gamma_{w,B}^j$ in every of which
the $q^{R-2}$-subset $\widehat{\Gamma}_{w,B}^j$ of affine points gives rise to hyperplanes containing~$B$.

 Thus, for the calculation of $\mathbb{S}_w$, the point $B$ will be counted $\#\bigcup\limits_{j=1}^{\BB_{w,1}}\widehat{\Gamma}_{w,B}^j$
  times.
 The same holds for all points of $\U_{w-1}\setminus\Pi_w$. Therefore,
 \begin{equation}\label{eq41:sum}
\mathbb{S}_w=\sum_{P\in\Pi_{w,1}}\delta_{w}(P)= \sum_{B\in\U_{w-1}\setminus\Pi_w}\#\bigcup\limits_{j=1}^{\BB_{w,1}}\widehat{\Gamma}_{w,B}^j.
 \end{equation}

 By \eqref{eq33:delta Alead}, \eqref{eq41:sum}, for $ \delta_{w}(A_{wR+1})$  we have
  \begin{equation}\label{eq41:delta Alead b}
 \delta_{w}(A_{wR+1})\ge
    \frac{\sum\limits_{P\in\Pi_{w,1}}\delta_{w}(P)}{q^{R-2}(q+1)}=
 \frac{\sum\limits_{B\in\U_{w-1}\setminus\Pi_w}\#\bigcup\limits_{j=1}^{\BB_{w,1}}\widehat{\Gamma}_{w,B}^j}{q^{R-2}(q+1)}.
 \end{equation}
 We denote
\begin{equation}\label{eq41:Gmin_def}
\G^{\min}_w\triangleq \min_{B\in\U_{w-1}\setminus\Pi_w}\#\bigcup\limits_{j=1}^{\BB_{w,1}}\widehat{\Gamma}_{w,B}^j.
\end{equation}
By \eqref{eq41:delta Alead b}, \eqref{eq41:Gmin_def},
\begin{equation}\label{eq41:aver3}
 \delta_{w}(A_{wR+1})\ge\frac{\G^{\min}_w\cdot\#(\U_{w-1}\setminus\Pi_w)}{q^{R-2}(q+1)}.
 \end{equation}

\begin{lemma}\label{lem4}
Let  $\BB_{w,1}=\binom{wR}{R-1}\le q+1$. The following holds:
 \begin{equation}\label{eq41:min>=a}
\G^{\min}_w\ge q^{R-3}\BB_{w,1}\left(q+\frac{1}{2}-\frac{1}{2}\BB_{w,1}\right)
=q^{R-3}\binom{wR}{R-1}\left(q+\frac{1}{2}-\frac{1}{2}\binom{wR}{R-1}\right).
   \end{equation}
\end{lemma}

\begin{proof}
By \eqref{eq31:mathfracBwi}, \eqref{eq32:suffic}, if $\BB_{w,1}-1\le q$ then, obviously,  $\#\Pi_{w,1}\ge1$.

For some $n$, we consider $n$ of the $q^{R-2}$-sets $\widehat{\Gamma}_{w,B}^j$ of \eqref{eq41:wideGamma}. All the sets are distinct. In fact, if
 $\widehat{\Gamma}_{w,B}^u = \widehat{\Gamma}_{w,B}^v , u\ne v$, then $\widehat{\Gamma}_{w,B}^u \subset  \Gamma_{w,B}^u\cap\Gamma_{w,B}^v$ that implies
 $q^{R-2} =  \#\widehat{\Gamma}_{w,B}^u < \#(\Gamma_{w,B}^u\cap\Gamma_{w,B}^v)=\theta_{R-3,q} $, contradiction.

 As $\widehat{\Gamma}_{w,B}^u $ and $\widehat{\Gamma}_{w,B}^v$ are the affine point sets of distinct $(R-2)$-spaces, they have at most $q^{R-3}$ points in common, i.e. $\#(\widehat{\Gamma}_{w,B}^u  \cap \widehat{\Gamma}_{w,B}^v) \leq q^{R-3}$.

Assume that $\#(\widehat{\Gamma}_{w,B}^u \cap \widehat{\Gamma}_{w,B}^v) = q^{R-3}$, for all pairs $(u,v)$, and that, in every set $\widehat{\Gamma}_{w,B}^j$, all the intersection points are distinct; it is the worst case  for $\#\bigcup\limits_{j=1}^{n}\widehat{\Gamma}_{w,B}^j$.

In every set $\widehat{\Gamma}_{w,B}^j$, the number of the affine point sets  intersecting it is $n-1$ and the number of the intersection points  is $(n-1)q^{R-3}$. As $\#\Gamma_{w,B}^j-(n-1)q^{R-3}=q^{R-2}-(n-1)q^{R-3}$ must be $\ge0$, the considered case is possible if $n-1\le q$.

In all $n$ sets $\widehat{\Gamma}_{w,B}^j$, the total
number of the intersection points  is $n(n-1)q^{R-3}$.  By above,   $\#\bigcup\limits_{j=1}^n\widehat{\Gamma}_{w,B}^j=nq^{R-2}-\frac{1}{2}n(n-1)q^{R-3}$ where $q^{R-2}=\#\widehat{\Gamma}_{w,B}^j$ and we need the factor~$\frac{1}{2}$ in order to calculate the meeting points exactly one time.

Finally, we put $n=\BB_{w,1}=\binom{wR}{R-1}$.
\end{proof}

\subsection{Estimates of sizes of $(R-1)$-saturating sets in $\PG(R,q)$. Upper bound on the length function $\ell_q(R+1,R)$}\label{subsec:estim}
Taking into account that $\Pc_{w,R}$ covers all points of $\Pi_w$, see Corollary \ref{cor32:b}, we have
\begin{equation}\label{eq42:Delta>=delta}
  \Delta_{w}(\Pc_{w,R})\ge\delta_{w}(A_{wR+1})+\#(\U_{w-1}\cap\Pi_w),
\end{equation}
where the  sign ``$\ge$'' is associated with the fact that the inclusion of the points $A_{wR+2},\ldots,A_{wR+R}$ can add new covered points outside $\Pi_w$.

\begin{lemma}
Let $\G^{\min}_w$ be as in \eqref{eq41:Gmin_def}, \eqref{eq41:aver3}
  For the number $\#\U_{w}$ of uncovered points after the $w$-th step of the iterative process, we have
\begin{equation}\label{eq42:Rwp1}
 \#\U_{w}\le q^R\prod_{m=1}^w\left(1- \frac{\G^{\min}_m}{q^{R-2}(q+1)}\right).
\end{equation}
\end{lemma}

\begin{proof}
By \eqref{eq31:Delta_def}, \eqref{eq42:Delta>=delta}, \eqref{eq41:aver3}, and Corollary \ref{cor32:b}, we have
 \begin{align*}
 &\Delta_{w}(\Pc_{w,R})=\#\U_{w-1}-\#\U_{w}=\#(\U_{w-1}\setminus\Pi_w)+\#(\U_{w-1}\cap\Pi_w)-\#\U_{w}\db\\
 &\ge\delta_{w}(A_{wR+1})+\#(\U_{w-1}\cap\Pi_w)\ge\frac{\G^{\min}_w\cdot\#(\U_{w-1}\setminus\Pi_w)}{q^{R-2}(q+1)}+\#(\U_{w-1}\cap\Pi_w),
  \end{align*}
 where $\G^{\min}_w\cdot\#(\U_{w-1}\setminus\Pi_w)$ is a lower bound of $\sum\limits_{B\in\U_{w-1}\setminus\Pi_w}\#\bigcup\limits_{j=1}^{\BB_{w,1}}\widehat{\Gamma}_{w,B}^j$, see \eqref{eq41:Gmin_def}, \eqref{eq41:aver3}. Therefore, $(\G^{\min}_w\cdot\#(\U_{w-1}\setminus\Pi_w))/q^{R-2}(q+1)$ is a lower bound of the number of the new covered points in $\U_{w-1}\setminus\Pi_w$. It follows that $\G^{\min}_w/q^{R-2}(q+1)\le1$, as the new covered points in the set $\U_{w-1}\setminus\Pi_w$ are a subset of it, that implies $(\G^{\min}_w\cdot\#(\U_{w-1}\setminus\Pi_w))/q^{R-2}(q+1)\le\#(\U_{w-1}\setminus\Pi_w)$. The summand $\#(\U_w\cap\Pi_w)$ takes into account that $\Pc_{w,R}$ covers all points of $\Pi_w$, see Corollary \ref{cor32:b}.

 As $\G^{\min}_w/q^{R-2}(q+1)\le1$ and $\#\U_{w-1}=\#(\U_{w-1}\setminus\Pi_w)+\#(\U_{w-1}\cap\Pi_w)$, we obtain
 \begin{equation*}
\Delta_{w}(\Pc_{w,R}) \ge\frac{\G^{\min}_w\cdot\#\U_{w-1}}{q^{R-2}(q+1)};
 \end{equation*}
 \begin{equation} \label{eq42:Rwp1_Rw}
    \#\U_{w}\le\#\U_{w-1}-\frac{\G^{\min}_w\cdot\#\U_{w-1}}{q^{R-2}(q+1)}=
    \#\U_{w-1}\left(1- \frac{\G^{\min}_w}{q^{R-2}(q+1)}\right).
 \end{equation}

As $R$ points of $\K_{0}$ are in general position, we have
\begin{equation*}
\#\U_{0}=\theta_{R,q}-\theta_{R-1,q}=q^R.
\end{equation*}
Starting from $\#\U_{0}$ and iteratively applying \eqref{eq42:Rwp1_Rw}, where $w$ is changed by $m$, we obtain the assertion.
\end{proof}

\begin{lemma}\label{lem42:b}
 Let  $\BB_{w,1}=\binom{wR}{R-1}\le q+1$.  The following holds:
\begin{equation}\label{eq42:estim0}
 1- \frac{\G^{\min}_w}{q^{R-2}(q+1)}<\exp\left(\frac{R^{R}}{R!}\left(-\frac{(w-1)^{R-1}}{q+1}+\frac{w^{2R-2}R^{R}}{2q^2\cdot R!}\right)\right).
\end{equation}
 \end{lemma}

 \begin{proof}
 By the inequality $1-x\le\exp(-x)$ and by \eqref{eq41:min>=a}, \eqref{eq31:mathfracBwi}, we have
\begin{align*}
 &  1- \frac{\G^{\min}_w}{q^{R-2}(q+1)}<\exp\left(-\frac{\G^{\min}_w}{q^{R-2}(q+1)}\right)\db\\
& <\exp\left(-\binom{wR}{R-1}\left(2q+1-\binom{wR}{R-1}\right)\frac{1}{2q(q+1)}\right)\db\\
& =\exp\left(-\binom{wR}{R-1}\frac{2q+1}{2q(q+1)}+\binom{wR}{R-1}^2\frac{1}{2q(q+1)}\right)\db\\
& <\exp\left(-\frac{(wR-R+2)^{R-1}}{(R-1)!\cdot(q+1)}+\left(\frac{(wR)^{R-1}}{(R-1)!}\right)^2\frac{1}{2q^2}\right)\db\\
& <\exp\left(-\frac{((w-1)R)^{R-1}}{(R-1)!\cdot(q+1)}+\frac{(wR)^{2R-2}}{2q^2((R-1)!)^2}\right)
\end{align*}
that implies the assertion.
 \end{proof}
Let $B_{2j}$ be a Bernoulli number \cite[Section~1.3]{HandbookMathForm}. We denote
\begin{align}\label{eq42:fwq def}
 &f_w(q,R)\triangleq\prod_{m=1}^w\left(1- \frac{\G^{\min}_m}{q^{R-2}(q+1)}\right),\db\\
 &\Dk_{w}^-(q,R)\triangleq\frac{(w-1)^R}{R(q+1)}+\frac{(w-1)^{R-1}}{2(q+1)}+
  \sum_{j=1}^{\left\lceil\frac{R-2}{2}\right\rceil}\frac{B_{2j}}{2j}\binom{R-1}{2j-1}\frac{(w-1)^{R-2j}}{q+1},\label{eq42:D- def}\db\\
  &\Dk_w^+(q,R)\triangleq \frac{w^{2R-1}}{2(2R-1)q^2}+\frac{w^{2R-2}}{4q^2}
  +\sum_{j=1}^{R-1}\frac{B_{2j}}{2j}\binom{2R-2}{2j-1}\frac{w^{2R-2j-1}}{2q^2}\label{eq42:D+ def}.
\end{align}
\begin{corollary}\label{cor42:}
  Let  $\BB_{w,1}=\binom{wR}{R-1}\le q+1$. We have
  \begin{align}\label{eq42:prod}
   & f_w(q,R)<\exp\left(\frac{R^{R}}{R!}\left(-\Dk_w^-(q,R)
 +\frac{R^{R}}{R!}\Dk_w^+(q,R)\right)\right).
  \end{align}
\end{corollary}

\begin{proof}
  By Lemma \ref{lem42:b} and \eqref{eq42:fwq def},
  \begin{align*}
    &f_w(q,R)<\prod_{m=1}^w\exp\left(\frac{R^{R}}{R!}\left(-\frac{(m-1)^{R-1}}{q+1}+\frac{m^{2R-2}R^{R}}{2q^2\cdot R!}\right)\right)\db\\
    &=\exp\left(\frac{R^{R}}{R!}\left(-\frac{1}{q+1}\sum_{u=1}^{w-1} u^{R-1}+\frac{R^R}{2q^2\cdot R!}\sum_{m=1}^w m^{2R-2}\right)\right).
     \end{align*}
     Then we use  \cite[Sections 1.2, 1.3]{HandbookMathForm}.
\end{proof}

Note that $B_2=1/6, B_4=B_8=-1/30, B_6=1/42,\ldots$

Corollary \ref{cor42:} allows to obtain exact relations for small $R$ and, also, asymptotic bounds for any $R$ when  $q$ tends to infinity.

\begin{lemma}\label{lem42:w suffic}
 Let $\BB_{w,1}=\binom{wR}{R-1}\le q+1$. To provide $\#\U_{w}\le q^Rf_q(w,R)\le R$, it is sufficient to take $w$ satisfying  the inequality
  \begin{equation}
  \Dk_w^-(q,R)-
 \frac{R^{R}}{R!}\Dk_w^+(q,R)  \ge \frac{R!}{R^{R-1}}\ln q. \label{eq42:w suffic}
\end{equation}
\end{lemma}

\begin{proof}
  If $q^Rf_q(w,R)\le 1$ then $q^Rf_q(w,R)\le R$. We
  take the logarithm of both the parts of the inequality $q^Rf_q(w,R)\le 1$ and use \eqref{eq42:prod}.
  \end{proof}

We will find the solution of the inequality \eqref{eq42:w suffic} in the form $w=\lceil \sqrt[R]{kq\ln q}\,\rceil$, where $k>0$  is a constant independent of $q$. For the convenience of research we write $w=\sqrt[R]{k q\ln q}+1$. The terms $\Dk_w^\mp(q,R)$  of \eqref{eq42:D- def} and \eqref{eq42:D+ def} take the form:
\begin{align}
 &\Dk_{w}^-(q,R)=\frac{kq\ln q}{R(q+1)}+\frac{(kq\ln q)^{1-1/R}}{2(q+1)}+
  \sum_{j=1}^{\left\lceil\frac{R-2}{2}\right\rceil}\frac{B_{2j}}{2j}\binom{R-1}{2j-1}\frac{(kq\ln q)^{1-2j/R}}{q+1},\notag\db\\
  &\Dk_w^+(q,R)< \frac{((k+1) q\ln q)^{2-1/R}}{2(2R-1)q^2}+\frac{((k+1) q\ln q)^{2-2/R}}{4q^2}\notag\db\\
  &+\sum_{j=1}^{R-1}\frac{B_{2j}}{2j}\binom{2R-2}{2j-1}\frac{((k+1) q\ln q)^{2-(2j-1)/R}}{2q^2},\label{eq42:D+}
\end{align}
that implies
\begin{align}\label{eq42:D q inf}
&\lim_{q\rightarrow\infty} \Dk_{w}^-(q,R)-\frac{kq\ln q}{R(q+1)}=0,~ \lim_{q\rightarrow\infty} \frac{q}{q+1}=1,~\lim_{q\rightarrow\infty} \Dk_{w}^-(q,R)
=\lim_{q\rightarrow\infty}\frac{k\ln q}{R}
;\db\\
&\lim_{q\rightarrow\infty} \Dk_{w}^+(q,R)=0.\label{eq42:D+ q inf}
\end{align}

\begin{lemma}\label{lem42:w>=}
Let $w= \sqrt[R]{kq\ln q}+1$, where $k>0$  is a constant independent of $q$. Then, for $q$ large enough, to provide $\#\U_{w}\le q^Rf_q(w,R)\le R$, it is sufficient to take $w$ satisfying  the inequality
  \begin{equation}
  w\ge\sqrt[R]{\frac{R!}{R^{R-2}}}\cdot\sqrt[R]{q\ln q}+1. \label{eq42:w suffic2}
\end{equation}
\end{lemma}
\begin{proof}
  From \eqref{eq31:mathfracBwi}
   \begin{equation*}
     \BB_{w,1}=\binom{R\sqrt[R]{kq\ln q}+R}{R-1}.
   \end{equation*}
    Recall that
     \begin{equation*}
     \binom{n}{k}= \frac{n \times (n-1) \times \dots \times (n-(k-1))}{k!}.
   \end{equation*}
   The numerator of such a fraction has $k$ factors. Then, for  $R$ fixed and $q$ large enough,
   \begin{equation*}
     \BB_{w,1}=\binom{R\sqrt[R]{kq\ln q}+R}{R-1} <  \binom{2R\sqrt[R]{kq\ln q}}{R-1} <  \frac{2^{R-1}R^{R-1}(kq\ln q)^{\frac{R-1}{R}}}{(R-1)!} < q+1
   \end{equation*}
as $k$ is a constant and the exponent of $q\ln q$ is less than 1.

Also, for $q$ large enough, \eqref{eq42:w suffic}, \eqref{eq42:D q inf}, and \eqref{eq42:D+ q inf} imply
   \begin{equation*}
     \frac{k\ln q}{R}\ge \frac{R!}{R^{R-1}}\ln q,~k  \ge \frac{R!}{R^{R-2}}.\qedhere
   \end{equation*}
\end{proof}
By Stirling's approximation of $R!$ \cite[Section 11.1.3.1]{HandbookMathForm}, \cite{Stirling}, we have
\begin{align*}
&\sqrt{2\pi R}\left(\frac{R}{e}\right)^R<\sqrt{2\pi R}\left(\frac{R}{e}\right)^R \cdot\sqrt[12R+1]{e}< R!<\sqrt{2\pi R}\left(\frac{R}{e}\right)^R\cdot \sqrt[12R]{e};\db\\
&R!\thickapprox\sqrt{2\pi R}\left(\frac{R}{e}\right)^R,~R\text{ is large enough}.
\end{align*}
After simple transformations this implies
\begin{align}\label{eq5:le cnew le}
&\frac{1}{e}<\frac{1}{e}\sqrt[2R]{2\pi R^5}< \sqrt[R]{\frac{R!}{R^{R-2}}}<\frac{1}{e}\sqrt[2R]{2\pi R^5} \cdot\sqrt[12R^2]{e};\db\\
&\sqrt[R]{\frac{R!}{R^{R-2}}}\thickapprox\frac{1}{e}\sqrt[2R]{2\pi R^5},~\frac{1}{R}\sqrt[R]{\frac{R!}{R^{R-2}}}\thickapprox\frac{1}{eR}\sqrt[2R]{2\pi R^5},~R\text{ is large enough}.\label{eq5:approximRbig}
\end{align}

\begin{lemma}\label{lem_lim_cnew}
\begin{equation}\label{eq1:lim cnew}
 \lim_{R\rightarrow\infty} c^{new}_{R}=\lim_{R\rightarrow\infty}\sqrt[R]{\frac{R!}{R^{R-2}}}=\frac{1}{e}\thickapprox0.3679.
\end{equation}
\end{lemma}
\begin{proof}
We use \eqref{eq5:le cnew le} and the fact that
\begin{equation*}
\lim_{R\rightarrow\infty} \sqrt[2R]{R^5}=\lim_{R\rightarrow\infty} e ^{\frac{5}{2R} \ln R } =1.
\end{equation*}

\end{proof}

\begin{theorem}\label{th32:bound}
Let $R\ge3$ be fixed. For the smallest size $s_q(R,R-1)$ of an $(R-1)$-saturating set in  $\PG(R,q)$ and  for the smallest
length of a $ q $-ary linear code with codimension (redundancy) $R+1$ and covering radius $R$ (i.e. for the length function $\ell_q(R+1,R)$)
the following asymptotic upper bounds hold:
  \begin{align}\label{eq32:bound}
   &\bullet~s_q(R,R-1) =\ell_q(R+1,R)\le\sqrt[R]{\frac{R!}{R^{R-2}}}\cdot q^{1/R}\cdot\sqrt[R]{\ln q}+1+R\db\\
    &=\sqrt[R]{\frac{R!}{R^{R-2}}}\cdot q^{1/R}\cdot\sqrt[R]{\ln q}+o(q),~q \text{ is large enough},~q\text{ is an arbitrary prime power};\notag\db\\
   &\bullet~\text{ if additionally }R\text{ is large enough, then in \eqref{eq32:bound} we have} \sqrt[R]{\frac{R!}{R^{R-2}}}\thicksim\frac{1}{e}\thickapprox0.3679.\label{eq32:boundbigR}
      \end{align}
\end{theorem}

\begin{proof}
  For \eqref{eq32:bound}, we use \eqref{eq1:s=ell}, Lemma \ref{lem42:w>=}, and add $R$ points to account for the last action of the iterative process, see Section \ref{subsec:iterproc}. For \eqref{eq32:boundbigR} we apply \eqref{eq1:lim cnew}.
\end{proof}
Note that for $r=R+1$, we have $(r-R)/R=1/R$, cf. the bounds \eqref{eq1:upbound}, \eqref{eq1:asympupbound}.

By construction we have the theorem:
\begin{theorem}
  The $(R-1)$-saturating $n$-set obtained by Construction B corresponds to an AMDS $[n,n-(R+1),R+1]_qR$ code with minimum distance $d=R+1$.
\end{theorem}

\section{Asymptotic upper bounds on the length function $\ell_q(tR+1,R)$}\label{sec:qmconc}

Proposition \ref{prop5_lift_r=Rt+1} is a variant of the lift-constructions ($q^m$-concatenating constructions) for covering codes \cite{Dav90PIT,DGMP-AMC,DavOst-IEEE2001,DavOst-DESI2010,DMP-AMC2021}, \cite[Section~5.4]{CHLS-bookCovCod}.

\begin{proposition}\label{prop5_lift_r=Rt+1}
\emph{\cite[Section 2, Construction $QM_1$]{DGMP-AMC}, \cite[Proposition 8.1]{DMP-AMC2021}}
Let an $ [n_{0},n_{0}-r_0]_{q}R$ code with $n_0\le q+1$ exist. Then there is an infinite family of
$[n,n-r]_{q}R$ codes with parameters
\begin{equation*}
  n=n_{0} q^m+R\theta_{m,q},~  r=r_0+Rm,~m\ge1.
\end{equation*}
\end{proposition}

We apply Proposition \ref{prop5_lift_r=Rt+1}, taking as $ [n_{0},n_{0}-r_0]_{q}R$ code the one corresponding to the $(R-1)$-saturating set obtained by Construction B, to construct infinite families of covering codes with growing codimension $r=tR+1$, $t\ge1$.
These families give asymptotic upper bounds  on the length function $\ell_q(tR+1,R)$  for growing $t\ge1$.

\begin{theorem}
Let $R\ge3$ be fixed. Let $t\ge1$. For the smallest
length of a $ q $-ary linear code with codimension (redundancy) $r=tR+1$ and covering radius $R$ (i.e. for the length function $\ell_q(tR+1,R)$) and for the smallest size $s_q(tR,R-1)$ of an $(R-1)$-saturating set in the projective space $\PG(tR,q)$
the following asymptotic upper bounds hold:
  \begin{align}
   &\bullet~\ell_q(tR+1,R) =s_q(tR,R-1)\le\sqrt[R]{\frac{R!}{R^{R-2}}}\cdot q^{(r-R)/R}\cdot\sqrt[R]{\ln q}+
   (1+R)q^{(r-R-1)/R}\notag\db\\
   &\hspace{0.4cm}+R(q^{(r-R-1)/R}-1)/(q-1)= \sqrt[R]{\frac{R!}{R^{R-2}}}\cdot q^{(r-R)/R}\cdot\sqrt[R]{\ln q}+o(q^{(r-R)/R}), \label{eq42:bound}\db\\
    &\hspace{0.6cm}r=tR+1, ~t\ge1,~ q\text{ is an arbitrary prime power},~q \text{ is large enough};\notag\db\\
    &\bullet~\text{ if additionally }R\text{ is large enough then in \eqref{eq42:bound} we have } \sqrt[R]{\frac{R!}{R^{R-2}}}\thicksim\frac{1}{e}\thickapprox0.3679.\label{eq42:newboundRbig}
  \end{align}
  The bounds are provided by an infinite family of $[n,n-r]_qR$ codes of the corresponding lengths.
\end{theorem}
\begin{proof}
 For $q$ large enough, it can be shown that $\sqrt[R]{\frac{R!}{R^{R-2}}}\cdot q^{1/R}\cdot\sqrt[R]{\ln q}+1+R<q+1$. Therefore we can take the code corresponding to the $(R-1)$-saturating set obtained by Construction B as the $ [n_{0},n_{0}-r_0]_{q}R$ code of Proposition \ref{prop5_lift_r=Rt+1}. In that, $r_0=R+1$, $m=t-1$, $r=R+1+(t-1)R$, $t-1=(r-R-1)/R$. This proves \eqref{eq42:bound}. For \eqref{eq42:newboundRbig} we apply \eqref{eq1:lim cnew}.
\end{proof}

\section{Properties of  $c^{new}_{R}=\sqrt[R]{\frac{R!}{R^{R-2}}}$. Comparison of new and known bounds}\label{sec:propCnew}

 In this section, we investigate properties of the constant  $c^{new}_{R}$ of the new asymptotic upper bound on $\ell_q(tR+1,R) =s_q(tR,R-1)$, and we show that it is essentially better than the one previously known in the literature. The limit $ \lim\limits_{R\rightarrow\infty} c^{new}_{R}=e^{-1}\thickapprox0.3679$ is noted in \eqref{eq1:lim cnew}.

The following lemma states that  $c^{new}_{R}$ is bounded from above and from below  by decreasing functions  of $R$.
\begin{lemma}\label{lem5:r=R+1}
Let $R\ge3$.
\begin{description}
  \item[(i)] The following functions are decreasing functions  of $R$:
  \begin{align*}
   \frac{1}{e}\sqrt[2R]{2\pi R^5},~\frac{1}{eR}\sqrt[2R]{2\pi R^5},~\frac{1}{e}\sqrt[2R]{2\pi R^5} \cdot\sqrt[12R^2]{e},~\frac{1}{eR}\sqrt[2R]{2\pi R^5} \cdot\sqrt[12R^2]{e}.
  \end{align*}

  \item[(ii)] The values $c^{new}_{R}=\sqrt[R]{\frac{R!}{R^{R-2}}}$ \eqref{eq1:c new} and $\frac{1}{R}c^{new}_{R}=\frac{1}{R}\sqrt[R]{\frac{R!}{R^{R-2}}}$ are bounded from above and from below  by decreasing functions  of $R$.
\end{description}
  \end{lemma}

\begin{proof}
  For $R\ge3$,  we have the derivative
\begin{align*}
  &\frac{d}{dR}\left(\sqrt[2R]{R^5}\right)=2.5\sqrt[2R]{R^5}  \frac{1-\ln R}{R^2}<0.
\end{align*}
So, $\sqrt[2R]{ R^5} $ is a decreasing function of $R$.
Obviously, also $\frac{1}{e}\sqrt[2R]{2\pi}$, $\sqrt[12R^2]{e}$, and $\frac{1}{R}$ are decreasing functions of $R$.
The product of decreasing functions is a decreasing function too. Finally, we use \eqref{eq5:le cnew le}.
\end{proof}

Examples of the value of $c^{new}_{R}$  \eqref{eq1:c new} and its lower and upper approximations of \eqref{eq5:le cnew le}, \eqref{eq5:approximRbig} are given in Table~\ref{tab1} from which one sees that the approximations are sufficiently convenient.

\begin{table}[htbp]
    \centering
    \caption{Values connected with the new upper bound and its approximations}
    \label{tab1}
\begin{tabular}{c|c|c|c}
  \hline
  $R$ &$\frac{1}{e}\sqrt[2R]{2\pi R^5}$&$c^{new}_{R}=\sqrt[R]{\frac{R!}{R^{R-2}}} \vphantom{H^{H^{H^{H^H}}}}$
  &$\frac{\sqrt[2R]{2\pi R^5}\cdot\sqrt[12R^2]{e} }{e}$\\\hline
  3 &1.24835051$\thickapprox 0.4161R$ &1.25992105$\thickapprox 0.4200R$& 1.25996299\\
  4 &1.10094468$\thickapprox 0.2752R$&1.10668192$\thickapprox 0.2767R$& 1.10669372\\
  5 &0.98857246$\thickapprox 0.1977R$&0.99186884$\thickapprox 0.1984R$ & 0.99187320\\
  6 &0.90458669$\thickapprox 0.1508R$&0.90668114$\thickapprox 0.1511R$& 0.90668307 \\
  7 &0.84050266$\thickapprox 0.1201R $&0.84193234$\thickapprox 0.1203R$& 0.84193331\\
  8 &0.79032802$\thickapprox 0.0988R$&0.79135723$\thickapprox 0.0989R$& 0.79135777\\
  9 &0.75009489$\thickapprox 0.0833R$&0.75086667$\thickapprox 0.0834R$& 0.75086699\\
  10&0.71715745$\thickapprox 0.0717R$&0.71775513$\thickapprox 0.0718R$& 0.71775533\\
  25&0.52657849$\thickapprox 0.0211R $&0.52664870$\thickapprox 0.0211R $& 0.52664871\\
   50&0.45565466$\thickapprox0.0091R $&0.45566985$\thickapprox 0.0091R$& 0.45566985\\
  100&0.41657808$\thickapprox0.0042R $&0.41658155$\thickapprox 0.0042R$& 0.41658155\\
  125&0.40816564$\thickapprox0.0033R $&0.40816781$\thickapprox 0.0033R$& 0.40816781\\
  150&0.40237807$\thickapprox0.0027R $&0.40237956$\thickapprox 0.0027R$& 0.40237956\\
  \hline
\end{tabular}
  \end{table}

If $R$ is large enough and the codimension $r=R+1$, the following lemma states that $c^{new}_{R}$ is about $R+1$ times better than the corresponding value previously known in the literature.

 \begin{lemma}
Let $R\ge4$, $r=R+1$. Then $\frac{c^{knw}_{R}}{c^{new}_{R}}=\frac{R^2}{R-1}\sqrt[R]{\frac{R-1}{R}}$ \eqref{eq1:knw/new} is an increasing function of $R$. Also,
\begin{equation}\label{eq5:r=R+1}
  \frac{c^{knw}_{R}}{c^{new}_{R}}=\frac{R^2}{R-1}\sqrt[R]{\frac{R-1}{R}}\thickapprox R+1\text{ if $R$ is large enough.}
\end{equation}
  \end{lemma}

\begin{proof}
  For $R\ge3$, we have the derivative
  \begin{align*}
   & \frac{d}{dR}\left(\frac{R^2}{R-1}\sqrt[R]{\frac{R-1}{R}} \right)=\frac{1}{R-1}\sqrt[R]{\frac{R-1}{R}}\left(R-\ln\frac{R-1}{R}-1\right)>0.
  \end{align*}

  Also,
  \begin{align}\label{limit2}
&\lim_{R\rightarrow\infty}\left(\frac{R^2}{R-1}\sqrt[R]{\frac{R-1}{R}}-R \right)=1.
  \end{align}
In fact,
  \begin{align*}
&\lim_{R\rightarrow\infty}\left(\frac{R^2}{R-1}\sqrt[R]{\frac{R-1}{R}}-R \right)=
\lim_{R\rightarrow\infty}\left(R\sqrt[R]{\frac{R^{(R-1)}}{(R-1)^{(R-1)}}}-R \right)\\
&=\lim_{R\rightarrow\infty}\left( \frac{R}{(R-1)^\frac{R-1}{R}}\left( R^\frac{R-1}{R} - (R-1)^\frac{R-1}{R} \right) \right)=\lim_{R\rightarrow\infty}\left( R^\frac{R-1}{R} - (R-1)^\frac{R-1}{R} \right) \\
&=\lim_{R\rightarrow\infty}\left( e^{\frac{R-1}{R} \ln(R)}  - e^{\frac{R-1}{R} \ln(R-1)} \right) = \lim_{R\rightarrow\infty} \left( e^{\frac{R-1}{R} \ln(R-1)} \left( e^{\frac{R-1}{R} \ln\left(\frac{R}{R-1}\right)}  - 1 \right)\right) \\
&=\lim_{R\rightarrow\infty} \left( (R-1)^\frac{R-1}{R} \left( e^{\frac{R-1}{R} \ln\left(\frac{R}{R-1}\right)}  - 1 \right)\right).
\end{align*}
As  $e^{f(x)} -1 \thickapprox f(x)$ if $f(x) \rightarrow 0$,
\begin{align*}
&\lim_{R\rightarrow\infty} \left( (R-1)^\frac{R-1}{R} \left( e^{\frac{R-1}{R} \ln\left(\frac{R}{R-1}\right)}  - 1 \right)\right)=
\lim_{R\rightarrow\infty} \left( (R-1)^\frac{R-1}{R} \frac{R-1}{R} \ln\left(\frac{R}{R-1}\right) \right) \\
&=\lim_{R\rightarrow\infty} \left( (R-1) \ln\left(\frac{R}{R-1}\right) \right) =\lim_{R\rightarrow\infty} \frac{ \ln\left(\frac{R}{R-1}\right)}{ \frac{1}{R-1}}.
\end{align*}
Applying l'H\^{o}pital's rule, we finally obtain
  \begin{align*}
& \lim_{R\rightarrow\infty} \frac{ \ln\left(\frac{R}{R-1}\right)}{ \frac{1}{R-1}}=
\lim_{R\rightarrow\infty} \frac{ \frac{1}{R-R^2}}{ \frac{1}{-(R-1)^2}}=1.
\end{align*}
    The assertion \eqref{eq5:r=R+1} follows from \eqref{limit2}.
\end{proof}

If $R$ is large enough and the codimension $r=tR+1, t\ge2$, the following lemma states that $c^{new}_{R}$ is about $9.32 R$ times better than the corresponding value previously known in the literature.

\begin{lemma}
Let $R\ge3, r=tR+1, t\ge2$. Then $\frac{c^{knw}_{R}}{c^{new}_{R}}=\frac{3.43R}{c^{new}_{R}}$ \eqref{eq1:knw/new}
  is bounded from above and from below by increasing functions  of $R$. Moreover,  if $R$ is large enough, then $c^{knw}_{R}/c^{new}_R\thickapprox 3.43e\thickapprox9.32R$.
\end{lemma}
\begin{proof}
  The first assertion follows from Lemma \ref{lem5:r=R+1}. The second one follows from \eqref{eq1:lim cnew}.
\end{proof}
\begin{table}[htbp]
    \centering
    \caption{Values connected with comparison of the new and known upper bounds;\newline
     $A=\frac{R}{R-1}\sqrt[R]{R(R-1)\cdot R!}$, $B=\frac{R^2}{R-1}\sqrt[R]{\frac{R-1}{R}}\thickapprox R+1$ \medskip}
    \label{tab2}
\begin{tabular}{r|r|r|r|r}
  \hline
  $R$ &$c^{new}_{R}=\vphantom{H^{H^{H^{H^H}}}}$  & $c^{knw}_{R}= $
  &$\frac{c^{knw}_{R}}{c^{new}_{R}}=B$&$\frac{c^{knw}_{R}}{c^{new}_{R}}=\frac{3.43R}{c^{new}_{R}}$\\
  &$\sqrt[R]{\frac{R!}{R^{R-2}}}$ &A~~~&$r=R+1$&$r=tR+1$\\
  &&&&$t\ge2$~~~~~\\\hline
  4 &1.1067&5.493&4.9632 &$12\thickapprox3.10R$\\
  5 &0.9919&5.929&5.9772&$17\thickapprox3.46R $\\
  6 &0.9067&6.333&6.9845&$23 \thickapprox3.78R$\\
  7 &0.8419&6.726&7.9888&$29 \thickapprox4.07R$\\
  8 &0.7914&7.116&8.9915&$35\thickapprox4.33R$\\
  9 &0.7509&7.504&9.9934&$41\thickapprox4.57R $\\
  10&0.7178&7.892&10.9947&$48\thickapprox 4.78R$\\
  25&0.5266&13.692&25.9992&$163\thickapprox6.51R$ \\
  50&0.4557&23.239&50.9998&$376\thickapprox7.53R$ \\
 100&0.4166&42.075&100.9999&$823\thickapprox8.23R$ \\
 125&0.4082&51.429&126.0000&$1050\thickapprox8.40R$ \\
 150&0.4024&60.759&151.0000&$1279\thickapprox 8.52R $\\
  \hline
\end{tabular}
  \end{table}

In Table \ref{tab2}, examples of the values connected with comparison of the new and known upper bounds are given.

From the results of this section, one sees that the new bounds are essentially better than the known ones.


\begin{thebibliography}{99}

\bibitem{AlderBruen-AMDS}
Alderson T.L., Bruen A.A.: Maximal AMDS codes.
Applicable Algebra Engineer. Commun. Comput. \textbf{19}(2), 87--98  (2008).
\url{https://doi.org/10.1007/s00200-008-0058-0}.

\bibitem{BDGMP-R2R3CC_2019}
Bartoli D., Davydov A.A., Giulietti M., Marcugini S., Pambianco F.:
New bounds for linear codes of covering radii 2 and 3.
Crypt. Commun. \textbf{11}(5), 903--920 (2019).
\url{https://doi.org/10.1007/s12095-018-0335-0}.

\bibitem{Bier}
Bierbrauer J.: Introduction to Coding Theory. 2$^{nd}$ edn.
Chapman and Hall/CRC Press (2017).

\bibitem{BoseBurt}
Bose R.C., Burton R.C.,
A characterization of flat spaces in a finite geometry and the uniqueness of the Hamming and McDonald codes.
J. Comb. Theory, \textbf{1}(1), 96--104 (1966).
\url{https://doi.org/10.1016/S0021-9800(66)80007-8}.

\bibitem{Handbook-coverings}
Brualdi R.A., Litsyn S., Pless V.: Covering radius. In: Pless V.S., Huffman W.C. (eds.)
Handbook of Coding Theory vol. \textbf{1}.
  Elsevier, Amsterdam, pp. 755--826 (1998).

\bibitem{Graismer-2023}
Chen H., Qu L., Li C., Lyu S., Xu L.: Generalized Singleton type upper bounds,
arXiv:2208.01138 [cs IT]  (2020).
\url{https://doi.org/10.48550/arXiv.2208.01138}.

\bibitem{CHLS-bookCovCod}
Cohen G., Honkala I., Litsyn S., Lobstein A.:
Covering Codes.
North-Holland Math. Library, \textbf{54}, Elsevier, Amsterdam, The Netherlands (1997).

\bibitem{Dav90PIT}
Davydov A.A.:
Construction of linear covering codes.
Probl. Inform. Transmiss. \textbf{26}(4), 317--331 (1990).
\url{http://iitp.ru/upload/publications/6833/ConstrCoverCodes.pdf}.

\bibitem{DGMP-AMC}
Davydov A.A., Giulietti M., Marcugini S., Pambianco F.:
Linear nonbinary covering codes and saturating sets in projective spaces. Adv. Math. Commun. \textbf{5}(1), 119--147 (2011).
\url{https://doi.org/10.3934/amc.2011.5.119}.

 \bibitem{DMP-R=tR2019}
Davydov A.A., Marcugini S., Pambianco F.:
New covering codes of radius $R$, codimension $tR$ and $tR+\frac{R}{2}$, and saturating sets in projective spaces.
Des. Codes Cryptogr. \textbf{87}(12), 2771--2792 (2019).
\url{https://doi.org/10.1007/s10623-019-00649-2}.

\bibitem{DMP-R=3Redun2019}
Davydov A.A., Marcugini S., Pambianco F.:
 New bounds for linear codes of covering radius 3 and 2-saturating sets in projective spaces. In:
Proc. 2019 XVI Int. Symp. Problems Redundancy Inform. Control
Systems (REDUNDANCY), Moscow, Russia, Oct. 2019, IEEE Xplore, pp. 47--52 (2020).
\url{https://doi.org/10.1109/REDUNDANCY48165.2019.9003348}.

  \bibitem{DMP-AMC2021}
Davydov A.A., Marcugini S., Pambianco F.:
 Upper bounds on the length function for covering codes with covering radius $R$ and codimension $tR+1$.
Adv. Math. Commun. \textbf{17}(1),  98--118 (2023).
\url{https://doi.org/10.3934/amc.2021074}.

\bibitem{DMP-rad3arXiv2023}
Davydov A.A., Marcugini S., Pambianco F.:
 New bounds for covering codes of radius 3 and
codimension $3t + 1$,  Adv. Math. Commun., to appear.
\url{https://doi.org/10.3934/amc.2023042}.

\bibitem{DavOst-IEEE2001}
Davydov A.A., \"{O}sterg{\aa}rd P.R.J.:
 Linear codes with covering radius $R=2,3$ and codimension $tR$.
IEEE Trans. Inform. Theory \textbf{47}(1), 416--421 (2001).
 \url{https://doi.org/10.1109/18.904551}.

\bibitem{DavOst-DESI2010}
Davydov A.A., \"{O}sterg{\aa}rd P.R.J.:
 Linear codes with covering radius 3.
Des. Codes Cryptogr. \textbf{54}(3), 253--271 (2010).
\url{https://doi.org/10.1007/s10623-009-9322-y}.

\bibitem{Boer-AMDS}
De Boer M.A.: Almost MDS codes. Des. Codes Cryptogr. \textbf{9}(2), 143--155 (1996).
\url{https://doi.org/10.1007/BF00124590}.

\bibitem{denaux}
Denaux L.:
Constructing saturating sets in projective spaces using subgeometries.
Des. Codes Cryptogr. \textbf{90}(9), 2113--2144 (2022).
\url{https://doi.org/10.1007/s10623-021-00951-y}.

\bibitem{denaux2}
Denaux L.:
 Higgledy-piggledy sets in projective spaces of small dimension.
Electron. J. Combin. \textbf{29}(3), Article \#P3.29 (2022).
\url{https://doi.org/10.37236/10736}.

\bibitem{DodunLang-NMDS}
Dodunekov S., Landgev I.: On near-MDS codes.
J. Geometry \textbf{54}(1), 30--43 (1995).
\url{https://doi.org/10.1007/BF01222850}.

\bibitem{EtzStorm2016}
Etzion T., Storme L.:
 Galois geometries and coding theory. Des. Codes Cryptogr. \textbf{78}(1), 311--350 (2016).
\url{https://doi.org/10.1007/s10623-015-0156-5}.

\bibitem{Giul2013Survey}
Giulietti M.:
 The geometry of covering codes: small complete caps and saturating sets in Galois spaces. In: Blackburn S. R., Holloway  R., Wildon M. (eds.)
 Surveys in Combinatorics 2013, Lecture Note Series \textbf{409}, pp. 51--90. London Math. Soc., Cambridge University Press, Cambridge, (2013).
\url{https://doi.org/10.1017/CBO9781139506748.003}.

\bibitem{GrahSlo}
Graham R.L., Sloane N.J.A.:
On the covering radius of codes. IEEE Trans. Inform. Theory \textbf{31}(1), 385--401 (1985).
\url{https://doi.org/10.1109/TIT.1985.1057039}.

\bibitem{HegNagy}
H\'{e}ger T., Nagy Z.L.:
 Short minimal codes and covering codes via strong blocking sets in projective spaces. IEEE Trans. Inform. Theory \textbf{68}(2), 881--890 (2022).
\url{https://doi.org/10.1109/TIT.2021.3123730}.

\bibitem{Hirs}
Hirschfeld J.W.P.: Projective Geometries over Finite Fields.
2$^{nd}$ edition, Oxford Univ. Press, Oxford, (1999).

\bibitem{HirsStor-2001}
Hirschfeld J.W.P., Storme L.:
The  packing problem in statistics, coding theory and finite projective spaces: Update 2001.
In: Blokhuis A., Hirschfeld J.W.P., Jungnickel  D., Thas  J.A. (eds.) Finite
    Geometries (Proc. 4th Isle of Thorns Conf., July 16-21, 2000),
    Develop. Math. \textbf{3}, pp. 201--246. Kluwer, Dordrecht, (2001).
\url{https://doi.org/10.1007/978-1-4613-0283-4_13}.

\bibitem{HufPless}
Huffman W.C., Pless V.S.: Fundamentals of Error-Correcting Codes.
Cambridge Univ. Press, (2003).

\bibitem{HandbookMathForm}
Jeffrey A., Dai H.H.:
Handbook of Mathematical Formulas and Integrals.
 4$^{th}$ edition.  Elsevier, Academic Press, Amsterdam, (2008).

\bibitem{LandSt}
Landjev I., Storme L.:
Galois geometry and coding theory.
In: De Beule J., Storme L. (eds.) Current Research Topics in Galois geometry, Chapter 8, pp. 187--214. NOVA Academic, New York (2011).

\bibitem{LobstBibl}
Lobstein A.: Covering radius, an online bibliography (2023).
Available at: \url{https://www.lri.fr/~lobstein/bib-a-jour.pdf}.

\bibitem{MWS}
MacWilliams F.J., Sloane N.J.A.:
The Theory of Error-Correcting Codes.
3$^{rd}$ edition. North-Holland, Amsterdam, The Netherlands, (1981).

\bibitem{AMDS-MenPelSal}
Meneghetti A., Pellegrini M., Sala M.:
A formula on the weight distribution of linear codes with applications to AMDS codes. Finite Fields Their Appl. \textbf{77},  article 101933,  (2022).
\url{https://doi.org/10.1016/j.ffa.2021.101933}.

\bibitem{Nagy}
Nagy Z.L.: Saturating sets in projective planes and hypergraph covers.
 Discrete Math. \textbf{341}(4), 1078--1083 (2018).
\url{https://doi.org/10.1016/j.disc.2018.01.011}.

\bibitem{Stirling}
Robbins H.: A remark on Stirling's formula. American Math Monthly
 \textbf{62}(1), 26--29 (1955).
\url{https://doi.org/10.2307/2308012}.

\bibitem{Struik}
Struik R.: Covering Codes.
Ph.D thesis. Eindhoven University of Technology, The Netherlands, (1994).
\url{https://doi.org/10.6100/IR425174}.


\end{thebibliography}
 \end{document}